\documentclass[a4paper,11pt]{article}
\usepackage{amsmath,amsthm,amssymb,enumitem,amscd,calc,titlesec}
\usepackage[bookmarks=false]{hyperref}
\usepackage{comment}

\renewcommand{\thesection}{\arabic{section}}
\titleformat{\section}{\Large\bf\boldmath}{\thesection.}{2ex}{}{}
\titlespacing{\section}{0ex}{2ex}{1ex}

\renewcommand{\thesubsection}{\arabic{section}.\arabic{subsection}}
\titleformat{\subsection}{\large\bf\boldmath}{\thesubsection.}{2ex}{}{}
\titlespacing{\section}{0ex}{1.5ex}{0.5ex}

\numberwithin{equation}{section}
{\theoremstyle{definition}\newtheorem{definition}{Definition}[section]

\newtheorem{remark}[definition]{Remark}
\newtheorem{remarkletter}{Remark}
}
\newtheorem{proposition}[definition]{Proposition}
\newtheorem{lemma}[definition]{Lemma}
\newtheorem{theorem}[definition]{Theorem}
\newtheorem{porism}[definition]{Porism}
\newtheorem{theoremletter}[remarkletter]{Theorem}
\newtheorem{corollaryletter}[remarkletter]{Corollary}

\newtheorem{step}{Step}[section]

\newcommand{\M}{\operatorname{M}}

\newcommand{\id}{\mathord{\operatorname{id}}}

\newcommand{\Tr}{\operatorname{Tr}}
\newcommand{\B}{\operatorname{B}}
\newcommand{\Ker}{\operatorname{Ker}}

\newcommand{\ot}{\otimes}

\newcommand{\dis}{\displaystyle}

\newcommand{\Om}{\Omega}
\newcommand{\supp}{\operatorname{supp}}

\newcommand{\otalg}{\otimes_{\text{\rm alg}}}

\newcommand{\Z}{\mathbb{Z}}
\newcommand{\R}{\mathbb{R}}
\newcommand{\C}{\mathbb{C}}
\newcommand{\recht}{\rightarrow}
\newcommand{\al}{\alpha}
\newcommand{\cF}{\mathcal{F}}
\newcommand{\actson}{\curvearrowright}

\newcommand{\ovt}{\mathbin{\overline{\otimes}}}
\newcommand{\cM}{\mathcal{M}}

\newcommand{\cK}{\mathcal{K}}
\newcommand{\cH}{\mathcal{H}}

\newcommand{\cU}{\mathcal{U}}

\newcommand{\eps}{\varepsilon}

\newcommand{\cS}{\mathcal{S}}
\newcommand{\dpr}{^{\prime\prime}}

\newcommand{\Ad}{\operatorname{Ad}}
\newcommand{\N}{\mathbb{N}}
\newcommand{\cZ}{\mathcal{Z}}
\newcommand{\cP}{\mathcal{P}}

\newcommand{\om}{\omega}

\newcommand{\cL}{\mathcal{L}}

\newcommand{\cR}{\mathcal{R}}
\newcommand{\cW}{\mathcal{W}}
\newcommand{\vphi}{\varphi}
\newcommand{\graph}{\operatorname{graph}}
\newcommand{\loc}{_{\text{\rm loc}}}
\newcommand{\locS}{_{\text{\rm loc-}\Sigma}}
\newcommand{\locXi}{_{\text{\rm loc-}\Xi}}
\newcommand{\invlimit}{\mathop{\underleftarrow{\operatorname{lim}}}}

\newcommand{\Hred}{\underline{H}}
\newcommand{\pijl}[1]{\overset{#1}{\longrightarrow}}
\newcommand{\clos}{\operatorname{cl}}
\newcommand{\image}{\operatorname{Im}}
\newcommand{\oIm}{\overline{\operatorname{Im}}}
\newcommand{\cC}{\mathcal{C}}
\newcommand{\betar}{\underline{\beta}}
\newcommand{\cE}{\mathcal{E}}

\newcommand{\covol}{\operatorname{covol}}
\newcommand{\tautil}{\tilde{\tau}}

\renewcommand{\leq}{\leqslant}
\renewcommand{\geq}{\geqslant}

\begin{document}

\begin{center}
{\boldmath\LARGE\bf  $L^2$-Betti  numbers  of locally compact groups and\vspace{0.5ex}\\ their cross section equivalence relations
}
\bigskip

{\sc by David Kyed\footnote{KU~Leuven, Department of Mathematics, Leuven (Belgium), david.kyed@wis.kuleuven.be \\
    Supported by ERC Starting Grant VNALG-200749}, Henrik Densing Petersen\footnote{University of Copenhagen, Department of Mathematical Sciences, Copenhagen (Denmark), hdp@math.ku.dk \\ Supported by the Danish National Research Foundation
through the Centre for Symmetry and Deformation (DNRF92)} and Stefaan Vaes\footnote{KU~Leuven, Department of Mathematics, Leuven (Belgium), stefaan.vaes@wis.kuleuven.be \\
    Supported by ERC Starting Grant VNALG-200749, Research
    Programme G.0639.11 of the Research Foundation --
    Flanders (FWO) and K.U.Leuven BOF research grant OT/08/032.}}
\end{center}

\bigskip
\begin{abstract}\noindent
We prove that the $L^2$-Betti numbers of a unimodular locally compact group $G$ coincide, up to a natural scaling constant, with the $L^2$-Betti numbers of the countable equivalence relation induced on a cross section of any essentially free ergodic probability measure preserving action of $G$. As a consequence, we obtain that the reduced and un-reduced $L^2$-Betti numbers of $G$ agree and that the $L^2$-Betti numbers of a lattice $\Gamma$ in $G$ equal those of $G$ up to scaling by the covolume of $\Gamma$ in $G$. We also deduce several vanishing results, including the vanishing of the reduced $L^2$-cohomology for amenable locally compact groups.
\end{abstract}

\section{Introduction}

The theory of $L^2$-Betti numbers, as well as related notions of $L^2$-invariants, provides a set of powerful invariants in geometry, topology and group theory, which are computable in many interesting cases. In \cite{At76}, Atiyah introduced $L^2$-Betti numbers for free cocompact group actions on manifolds. This was generalized by Connes \cite{Co79} to a set of invariants of measured foliations. For arbitrary countable groups $\Gamma$, the $L^2$-Betti numbers $\beta^n_{(2)}(\Gamma), n\in \mathbb{N}$, were defined by Cheeger and Gromov in \cite{CG85}.

Gaboriau, in \cite{Ga01}, defined the $L^2$-Betti numbers $\beta^n_{(2)}(\cR)$ of an arbitrary countable probability measure preserving (pmp) equivalence relation. As a consequence, the $L^2$-Betti numbers of a measured foliation with contractible leaves only depend on the associated equivalence relation. Furthermore, Gaboriau proves that $\beta^n_{(2)}(\Gamma) = \beta^n_{(2)}(\cR_\Gamma)$ for every countable group $\Gamma$ with an essentially free ergodic pmp action $\Gamma \actson (X,\mu)$ and orbit equivalence relation $\cR_\Gamma$. So $L^2$-Betti numbers are invariant under orbit equivalence and, as also shown in \cite{Ga01}, scale under measure equivalence of groups by the compression constant of the measure equivalence.

$L^2$-Betti numbers have been generalized further to a variety of different settings, see \cite{Sa03,CS04,Ky06}, and we refer to \cite{Lu02} for an extensive monograph on the subject.

In all cases, $L^2$-Betti numbers are defined as the Murray-von Neumann dimension of certain (co)homology modules  with coefficients in the group von Neumann algebra $L \Gamma$, or the von Neumann algebra $L \cR$ of a countable probability measure preserving (pmp) equivalence relation. By the work of L\"{u}ck \cite{Lu97}, one can define the dimension of an arbitrary (purely algebraic) module over a tracial von Neumann algebra $(M,\tau)$ and this provides a reinterpretation of the $L^2$-Betti numbers of a countable group $\Gamma$ by means of the formula
$$\beta_{(2)}^n(\Gamma) = \dim_{L \Gamma} H^n(\Gamma,\ell^2(\Gamma)) \; .$$
L\"{u}ck's dimension theory was extended to von Neumann algebras equipped with semifinite traces in \cite[Appendix B]{Pe11} (see also Appendix \ref{appA} in this article for details on dimension theory). Hence $L^2$-Betti numbers of unimodular locally compact second countable (lcsc) groups could be defined in \cite[Section 3.1]{Pe11} by the formula
$$\beta_{(2)}^n(G) = \dim_{L G} H^n(G,L^2(G)) \; .$$
This definition was motivated in part by the following two well-known facts for discrete groups.
\begin{enumerate}
\item If $\Lambda \leq \Gamma$ is an inclusion of countable groups with finite index $[\Gamma:\Lambda]$ then the $L^2$-Betti numbers scale according to the formula $\beta^n_{(2)}(\Gamma)=[\Gamma:\Lambda]^{-1}\beta_{(2)}^n(\Lambda)$. Cf.~\cite[Proposition 2.6]{CG85}.
\item If $\Gamma$ and $\Lambda$ are lattices in a common second countable, locally compact topological group $G$ then the $L^2$-Betti numbers of $\Lambda$ and $\Gamma$ are proportional; more precisely, one has
\[
\beta_{(2)}^n(\Gamma)=\frac{\covol(\Gamma)}{\covol(\Lambda)}\beta_{(2)}^n(\Lambda)
\]
for all $n\geq 0$. This is a special case of Gaboriau's theorem about measure equivalence invariance of $L^2$-Betti numbers \cite[Th\'{e}or\`{e}me 6.3]{Ga01}.
\end{enumerate}
With these observations in mind, if $G$ is a unimodular lcsc group and $H \leq G$ is a closed unimodular subgroup of finite covolume, it is a very natural question whether
\begin{align}\label{intro-scaling-formula}
\beta_{(2)}^n(G)=\frac{1}{\covol(H)}\beta_{(2)}^n(H).
\end{align}
In \cite[Theorems 4.8 and 5.9]{Pe11}, such a result was proved for cocompact lattices and also in the case when $G$ is totally disconnected. One of our main results is to prove \eqref{intro-scaling-formula} in its full generality.

Our method is based on an observation, following \cite{Fo74}, that in a measurable sense every unimodular lcsc group $G$ admits a cocompact lattice. More precisely, for every essentially free ergodic pmp action $G \actson (X,\mu)$, there exists a cocompact cross section $Y \subset X$ (see Sections \ref{subsec.statements} and \ref{subsec.cross-section} for terminology). This implies that the restriction of the orbit equivalence relation of $G \actson X$ to $Y$ is a countable pmp equivalence relation $\cR$ and that there exists a compact subset $K \subset G$ such that $K \cdot Y$ is conegligible in $X$. Our main theorem says that the $L^2$-Betti numbers $\beta_{(2)}^n(G)$ of $G$ are proportional to the $L^2$-Betti numbers $\beta^n_{(2)}(\cR)$ of the equivalence relation $\cR$, in the sense of Gaboriau \cite{Ga01}. The proportion between the two is given by a natural constant that we call the covolume of $Y$.

We can then reduce several questions about $L^2$-Betti numbers of $G$ to known results for $L^2$-Betti numbers of countable pmp equivalence relations. In this way, we prove that the reduced and unreduced $L^2$-Betti numbers of $G$ coincide and we establish several vanishing results. This includes the vanishing of all $L^2$-Betti numbers and of the reduced cohomology groups $\Hred^n(G,L^2(G))$ whenever $G$ admits a noncompact amenable closed normal subgroup, in particular when $G$ is noncompact and amenable.
This extends a well-known result of Cheeger-Gromov \cite{CG85} for countable groups. For connected amenable groups the vanishing of reduced $L^2$-cohomology in degree one was proved essentially by Delorme in \cite{De77} (see also \cite{Ma04}). In contrast to the proof of Delorme, our more general result follows directly from the vanishing of $L^2$-Betti numbers for the (unique) amenable ergodic II$_1$ equivalence relation.

\subsection{Notation and standing assumptions}

In what follows, all topological groups are implicitly assumed to be Hausdorff and we will use the abbreviation lcsc for `locally compact second countable'. A nonsingular action of a lcsc group $G$ on a standard measure space $(X,\mu)$ is an action of $G$ on the set $X$ such that the map $G \times X \recht X : (g,x) \mapsto g \cdot x$ is Borel and such that $\mu(g \cdot A) = 0$ whenever $A \subset X$ is a Borel set of measure zero. We say that the action is pmp (probability measure preserving) if $\mu$ is a probability measure and $\mu(g \cdot A) = \mu(A)$ for all $g \in G$ and all Borel sets $A \subset X$.

When $G \actson (X,\mu)$ is a nonsingular action, one can show that the set of points $x \in X$ having a trivial stabilizer is a Borel set (see e.g.\ \cite[Lemma 10]{MRV11} for a proof of this well known result). If this Borel set is conegligible, we say that the action is essentially free. For later use, we recall the following.

\begin{remark}\label{exist-action}
Every lcsc group $G$ admits an essentially free ergodic (even mixing) pmp action $G \actson (X,\mu)$. Indeed, it suffices to denote by $(X_0,\mu_0)$ the Gaussian probability space that corresponds to the real Hilbert space $L^2_\R(G)$. The Gaussian action $G \actson (X_0,\mu_0)$ is pmp and faithful. Since the Koopman representation on $L^2(X_0,\mu_0) \ominus \C 1$ is a multiple of the regular representation of $G$, the action $G \actson (X_0,\mu_0)$ is mixing. The diagonal action of $G$ on the infinite direct product $(X,\mu) = (X_0,\mu_0)^\N$ is then essentially free, mixing and pmp (see \cite[Proposition 1.2]{AEG93} for details).
\end{remark}

\subsection{Statement of the main results}\label{subsec.statements}

Let $G$ be a lcsc group and $G \actson (X,\mu)$ an essentially free pmp action.
We call a Borel set $Y \subset X$ a \emph{cross section} of $G \actson (X,\mu)$ if there exists a neighborhood of the identity $\cU \subset G$ such that the map $\theta: \cU \times Y \recht X : (g,y) \mapsto g \cdot y$ is injective and such that $\mu(X - G \cdot Y) = 0$.

We recall the following classical results and refer to Section \ref{subsec.cross-section} for a more detailed explanation and proofs. Every essentially free pmp action admits a cross section. Then
$$\cR := \{(y,y') \in Y \times Y \mid y \in G \cdot y'\}$$
is a countable Borel equivalence relation on $Y$, which is called the \emph{cross section equivalence relation}. Assume that $G$ is unimodular and fix a Haar measure $\lambda$ on $G$. Then $Y$ is equipped with a unique $\cR$-invariant probability measure $\nu$ satisfying $\theta_*(\lambda \times \nu) = \covol Y \cdot \mu_{|\cU \cdot Y}$, for some positive scaling factor $\covol Y$.

Our main result relates the $L^2$-Betti numbers of $G$ to those of $\cR$ by means of the following theorem. The precise definition for the $L^2$-Betti numbers of $G$, resp.\ $\cR$, is given in Sections \ref{lcgrp-cohomology-section} and \ref{sec.cohom-equiv-rel}.

\begin{theoremletter}\label{thm.mainA}
Let $G$ be a lcsc unimodular group and $G \actson (X,\mu)$ an essentially free ergodic pmp action. For every cross section $Y \subset X$ with corresponding cross section equivalence relation $\cR$ and for every $n \in \N$, we have
$$\beta^{n}_{(2)}(G) = \betar^{n}_{(2)}(G) = \frac{1}{\covol Y} \, \beta_n^{(2)}(\cR) \; .$$
\end{theoremletter}

If $H$ is a closed subgroup of the lcsc unimodular group $G$, then $G/H$ admits a $G$-invariant measure if and only if $H$ is unimodular (see e.g.\ \cite[Corollary B.1.7]{BHV08}). In that case, the $G$-invariant measure on $G/H$ is unique up to scaling and once we have fixed Haar measures $\lambda_G$ and $\lambda_H$, there is a canonical choice $\lambda_{G/H}$ satisfying
\begin{equation}\label{eq.invariant-quotient}
\Phi_*(\lambda_{G/H} \times \lambda_H) = \lambda_G
\end{equation}
where $\Phi(gH,h) = \theta(gH) h$ and $\theta : G/H \recht G$ is any Borel cross section. We denote $\covol H := \lambda_{G/H}(G/H)$.

\begin{theoremletter}\label{thm.mainB}
Let $G$ be a lcsc unimodular group and $H < G$ a closed unimodular subgroup of finite covolume. Given fixed Haar measures on $G$ and $H$, we have
$$\beta^n_{(2)}(G) = \frac{1}{\covol H} \, \beta^n_{(2)}(H) \quad\text{for all}\;\; n \geq 0 \; .$$
In particular, if $\Gamma$ is a lattice in the lcsc group $G$, then $\beta^n_{(2)}(G) = \covol (\Gamma)^{-1} \beta_n^{(2)}(\Gamma)$ for all $n \geq 0$.
\end{theoremletter}

For the following result, note that a closed normal subgroup of a lcsc unimodular group is again unimodular.

\begin{theoremletter}\label{thm.mainC}
Let $G$ be a lcsc unimodular group.
\begin{enumerate}
\item If $G$ is compact with Haar measure $\lambda$, then $\beta^0_{(2)}(G) = \lambda(G)^{-1}$ and $\beta^n_{(2)}(G) = 0$ for all $n \geq 1$.
\item If $G$ is noncompact and amenable, then $\beta^n_{(2)}(G) = 0$ for all $n \geq 0$.
\item Let $G$ be a lcsc unimodular group and $H \lhd G$ a closed normal subgroup. If $d \geq 0$ and $\beta^n_{(2)}(H) = 0$ for all $0 \leq n \leq d$, then
    $\beta^n_{(2)}(G) = 0$ for all $0 \leq n \leq d$.
\item Let $G$ be a lcsc unimodular group and $H \lhd G$ a closed normal subgroup such that $G/H$ is noncompact. If $d \geq 0$,  $\beta^n_{(2)}(H) = 0$ for all $0 \leq n \leq d$ and  $\beta^{d+1}_{(2)}(H) < \infty$, then $\beta^n_{(2)}(G) = 0$ for all $0 \leq n \leq d+1$.
\end{enumerate}
\end{theoremletter}

In Proposition \ref{prop.finite-first-Betti} we will show that $\beta^1_{(2)}(G) < \infty$ whenever $G$ is compactly generated, in particular when $G$ is connected. We therefore obtain the following corollary.

\begin{corollaryletter}\label{cor.corD}
Let $G$ be a lcsc unimodular group. If $G$ admits a closed normal subgroup $H$ such that $H$ is compactly generated and such that both $H$ and $G/H$ are noncompact, then $\beta^1_{(2)}(G) = 0$.
\end{corollaryletter}

The following corollary, which follows directly from Theorem \ref{thm.mainB} and Theorem \ref{thm.mainC} gives an alternative approach to the vanishing of $L^2$-Betti numbers in \cite[Remark 1.9]{BFS12}.

\begin{corollaryletter}\label{cor.corE}
Let $G$ be a lcsc unimodular group. If $G$ admits a noncompact, amenable, closed, normal subgroup then $\beta^n_{(2)}(G) = 0$ for all $n\geq  0$, whence in particular the $L^2$-Betti numbers vanish for any lattice in $G$.
\end{corollaryletter}

Notice that by \ref{vanishing-porism} below, the vanishing of the $n$-th $L^2$-Betti number $\beta^n_{(2)}(G)$ is equivalent with the vanishing of the $n$-th \emph{reduced} cohomology group $\Hred^n(G,L^2(G))$. So the vanishing results \ref{thm.mainC}-\ref{cor.corD}-\ref{cor.corE} can also be viewed as vanishing results for reduced cohomology groups.

Finally note that a combination of Corollary \ref{cor.corE} with the structure theory of lcsc groups and the K\"{u}nneth formula allows in principle to reduce all computations of $L^2$-Betti numbers $\beta^n_{(2)}(G)$ to computations where $G$ is totally disconnected. We refer to \cite[Chapter 7]{Pe11} for details.

\section{\boldmath Cohomology and $L^2$-Betti numbers of locally compact groups}\label{lcgrp-cohomology-section}

In this section we fix the definitions of cohomology and $L^2$-Betti numbers that we will use for locally compact groups.

\subsection{\boldmath $L^2$-Betti numbers of locally compact unimodular groups}

Let $G$ be a lcsc group, $P$ an algebra and $\cH$ a Fr\'{e}chet space. Denote by $S(\cH)$ the set of continuous seminorms on $\cH$.
\begin{itemize}
\item We call $\cH$ a left Fr\'{e}chet $G$-module if $\cH$ is equipped with a left action of $G$ by linear maps such that $G \times \cH \recht \cH : (g,\xi) \mapsto g \cdot \xi$ is continuous.
\item We call $\cH$ a right Fr\'{e}chet $P$-module if $\cH$ is a right $P$-module and if for every $a \in P$, the map $\cH \recht \cH : \xi \mapsto \xi \cdot a$ is continuous.
\item We call $\cH$ a Fr\'{e}chet $G$-$P$-bimodule if $\cH$ is both a left Fr\'{e}chet $G$-module and a right Fr\'{e}chet $P$-module and if  the two actions commute.
\item If $X$ is a lcsc space, then the vector space $C(X,\cH)$ of continuous functions from $X$ to $\cH$ is again a Fr\'{e}chet space, using the seminorms
$$\xi \mapsto \sup_{x \in K} p(\xi(x)) \quad\text{for all}\;\; K \subset X \;\;\text{compact}\;, \;\; p \in S(\cH) \; .$$
If $\cH$ is a right Fr\'{e}chet $P$-module, then $C(X,\cH)$ naturally is a right Fr\'{e}chet $P$-module.
\end{itemize}

By a \emph{complex of Fr\'{e}chet spaces} we mean a sequence $\cC\colon \cH_0 \pijl{d_0} \cH_1 \pijl{d_1} \cH_2 \pijl{d_2} \cdots$ such that each $\cH_n$ is a Fr\'{e}chet space and the maps $d_n \colon \cH_n \recht \cH_{n+1}$ are continuous and satisfy $d_{n+1} \circ d_n = 0$ for all $n \geq 0$.

For $n\geq 0$ the $n$'th cohomology of $\cC$, respectively the $n$'th reduced cohomology of $\cC$, are defined as

\begin{align*}
& H^0(\cC) := \Ker d_0 \quad\text{and}\quad H^n(\cC) := \frac{\Ker d_n}{\image d_{n-1}} \;\;\text{for all}\;\; n \geq 1 \;\; , \quad\text{respectively} \\
& \Hred^0(\cC) := \Ker d_0 \quad\text{and}\quad \Hred^n(\cC) := \frac{\Ker d_n}{\clos(\image d_{n-1})} \;\;\text{for all}\;\; n \geq 1 \; .
\end{align*}
If the $\cH_n$ are Fr\'{e}chet $P$-modules and the maps $d_n$ are $P$-linear, we call $\cC$ a \emph{complex of Fr\'{e}chet $P$-modules.} Then, $H^n(\cC)$ and $\Hred^n(\cC)$ are $P$-modules, with the latter being a Fr\'{e}chet $P$-module. When the $\cH_n$ are Fr\'{e}chet $G$-$P$-bimodules and the maps $d_n$ are $G$-$P$-linear, we call $\cC$ a \emph{complex of Fr\'{e}chet $G$-$P$-bimodules.}

\begin{definition}\label{def.bar-res}
Let $G$ be a lcsc group, $P$ an  algebra and $\cH$ a Fr\'{e}chet $G$-$P$-bimodule. Then $H^n(G,\cH)$, is defined as the $n$-th cohomology group of the complex of Fr\'{e}chet $P$-modules
\begin{equation}\label{eq.bar-res}
\cH \pijl{d_0} C(G,\cH) \pijl{d_1} C(G^2,\cH) \pijl{d_2} \cdots
\end{equation}
where the coboundary maps $d_n$ are defined by
\begin{multline}\label{eq.bar-coboundary}
(d_n \xi)(g_0,\ldots,g_n) = g_0 \cdot \xi(g_1,\ldots,g_n) - \xi(g_0 g_1,g_2,\ldots,g_n) + \cdots \\
\cdots + (-1)^n \xi(g_0,\ldots,g_{n-2},g_{n-1} g_n)  + (-1)^{n+1} \xi(g_0,\ldots,g_{n-1}) \; .
\end{multline}
Note that $H^n(G,\cH)$ naturally is a right $P$-module.

Further, $\Hred^n(G,\cH)$ is defined as the $n$-th reduced cohomology group of the complex \eqref{eq.bar-res} of Fr\'{e}chet $P$-modules. Also $\Hred^n(G,\cH)$ is a right $P$-module.
\end{definition}

Following \cite{Pe11}, we now define the $L^2$-Betti numbers, $\beta^n_{(2)}(G)$, of a lcsc \emph{unimodular} group $G$ as the Murray-von Neumann dimension of the cohomology groups $H^n(G,L^2(G))$. Recall that the group von Neumann algebra $LG$ of a lcsc group $G$ is defined as the von Neumann algebra generated by the left regular representation of $G$ on $L^2(G)$. If $G$ is unimodular with a fixed Haar measure $\lambda$, the von Neumann algebra $LG$ is equipped with a natural semifinite trace. It is the unique normal semifinite faithful trace $\Tr$ on $LG$ satisfying
$$\Tr\Bigl(\int_G f(g) u_g \; dg\Bigr) = f(e)$$
for every continuous compactly supported function $f : G \recht \C$. Note that $L^2(G)$ naturally is an $LG$-$LG$-bimodule, using the left and the right regular representations.

Whenever $(N,\Tr)$ is a von Neumann algebra equipped with a normal semifinite faithful trace, one can define the dimension $\dim_N \cH$ of an arbitrary $N$-module $\cH$, see Definition \ref{def.dim-semifinite} in Appendix \ref{appA} on dimension theory.

\begin{definition}
Let $G$ be a lcsc unimodular group. We define for all $n \geq 0$,
$$\beta^{n}_{(2)}(G) := \dim_{LG} H^n(G,L^2(G)) \quad\text{and}\quad \betar^n_{(2)}(G) := \dim_{LG} \Hred^n(G,L^2(G)) \; .$$
\end{definition}

When $G$ is discrete this definition agrees with  the  standard definitions, see e.g.~\cite[Section 2]{PT07}.

The main purpose of this article is to prove that, up to a natural rescaling, the $L^2$-Betti numbers $\betar^n_{(2)}(G)$ are equal to the $L^2$-Betti numbers of the cross section equivalence relation associated with an arbitrary free ergodic pmp action of $G$ (see Theorem \ref{thm.mainA}). As a byproduct, we get that $\beta^{n}_{(2)}(G) = \betar^n_{(2)}(G)$ for all lcsc unimodular groups $G$. This equality was already shown in \cite[Theorem 5.6]{Pe11} whenever $G$ is totally disconnected or admits a cocompact lattice.

\subsection{Basic cohomology theory for locally compact groups}

To identify $H^n(G,L^2(G))$ with a cohomology theory of the associated cross section equivalence relations, we need some basic tools from homological algebra. The cohomology theory for locally compact groups defined by continuous cochains was first considered by Mostow in \cite{Mo61}. Standard references are the monographs \cite{BW80,Gu80}. For the convenience of the reader, we list the needed properties in the rest of this section.

\begin{definition}\label{def.strong-exact-acyclic}
Let $G$ be a lcsc group and $P$ an algebra.
\begin{enumerate}
\item A complex of  right Fr\'{e}chet $P$-modules $\cH_0 \pijl{d_0} \cH_1 \pijl{d_1} \cH_2 \pijl{d_2} \cdots$ is called \emph{exact} if $\Ker d_n = \image d_{n-1}$ for all $n \geq 1$. It is called \emph{strongly exact} if there exist continuous $P$-linear maps $S_n : \cH_n \recht \cH_{n-1}$, for all $n \geq 1$, such that
    $$S_{n+1} \circ d_n + d_{n-1} \circ S_n = \id_{\cH_n} \quad\text{for all}\;\; n \geq 1 \; .$$
\item A Fr\'{e}chet $G$-$P$-bimodule $\cH$ is called \emph{strongly acyclic} if the complex of $P$-modules
$$0 \longrightarrow \cH^G \longrightarrow \cH \pijl{d_0} C(G,\cH) \pijl{d_1} C(G^2,\cH) \pijl{d_2} \cdots$$
is strongly exact. Here $\cH^G$ denotes the Fr\'{e}chet $P$-submodule of $\cH$ that consists of the $G$-fixed points in $\cH$, and the coboundary maps $d_n$, $n \geq 0$, are the ones given in \eqref{eq.bar-coboundary}.
\end{enumerate}
\end{definition}

The following proposition is standard (cf.\ \cite[Proposition 2.9]{Bl77}) and we leave the proof as an exercise.

\begin{proposition}\label{prop.resolution}
Let $G$ be a lcsc group, $P$ an algebra and $\cH$ a Fr\'{e}chet $G$-$P$-bimodule. Assume that
\begin{align}\label{s-exact-cplx}
0 \longrightarrow \cH \pijl{d_{-1}} \cH_0 \pijl{d_0} \cH_1 \pijl{d_1} \cdots
\end{align}
is a complex of Fr\'{e}chet $G$-$P$-bimodules that is strongly exact as a complex of Fr\'{e}chet $P$-modules. Assume that for all $n \geq 0$, the Fr\'{e}chet $G$-$P$-bimodule $\cH_n$ is strongly acyclic.
Let $\cC$ be the complex of Fr\'{e}chet $P$-modules given by
$$\cH_0^G \pijl{d_0} \cH_1^G \pijl{d_1} \cH_2^G \pijl{d_2} \cdots \; .$$
Then there are natural $P$-linear isomorphisms
$$H^n(G,\cH) \cong H^n(\cC) \quad\text{and}\quad \Hred^n(G,\cH) \cong \Hred^n(\cC) \quad\text{for all}\;\; n \geq 0 \; .$$
\end{proposition}

We will apply Proposition \ref{prop.resolution} to $G$-$P$-bimodules $\cH_n$ of the form $\cH_n = L^2\loc(G,\cK_n)$. So we first recall some Fr\'{e}chet-valued integration theory.

Let $(Z,\eta)$ be a standard Borel space equipped with a $\sigma$-finite measure. Let $\cH$ be a Fr\'{e}chet space. Denote by $S(\cH)$ the set of continuous seminorms on $\cH$. For every $1 \leq p < +\infty$, we denote by $L^p(Z,\cH)$ the space of functions $f\colon Z \recht \cH$ that are strongly Borel (in the sense that inverse images of open sets are Borel sets) and that have the property that $q \circ f \in L^p(Z)$ for every $q \in S(\cH)$, where we implicitly identify functions that are equal a.e. The family of seminorms $f \mapsto \|q \circ f\|_p$, $q \in S(\cH)$, turns $L^p(Z,\cH)$ into a Fr\'{e}chet space. All $f \in L^1(Z,\cH)$ have an integral $\int_Z f \, d\eta$ in $\cH$ and the map $f \mapsto \int_Z f \, d\eta$ is a continuous linear map from $L^1(Z,\cH)$ to $\cH$.

In order to define locally square integrable functions, assume that the standard $\sigma$-finite measure space $(Z,\eta)$ comes with an increasing sequence of Borel sets $Z_n \subset Z$ such that $\bigcup_n Z_n$ has complement of measure zero. We define $L^p\loc(Z,\cH)$ as the space of strongly Borel functions $f\colon Z\to \cH$ such that $f_{|Z_n}$ belongs to $L^p(Z_n,\cH)$ for every $n$. Using the seminorms $f \mapsto \|(q \circ f)_{|Z_n}\|_p$, for all $q \in S(\cH)$ and all $n$, we turn $L^p\loc(Z,\cH)$ into a Fr\'{e}chet space.
Observe that if $P$ is an algebra and $\cH$ is a right Fr\'{e}chet $P$-module, then also $L^p\loc(Z,\cH)$ is a right Fr\'{e}chet $P$-module.

Note that $L^p\loc(Z,\cH)$ only depends on the choice of the sequence $(Z_n)$ up to cofinality: if $(Z'_n)$ is another increasing sequence of Borel sets whose union has complement of measure zero, and if for every $n$, there exists an $m$ such that $Z_n \subset Z'_m$ and $Z'_n \subset Z_m$, then the Fr\'{e}chet spaces $L^p\loc(Z,\cH)$ defined w.r.t.\ $(Z_n)$ and $(Z'_n)$ coincide.

If $G$ is a lcsc group, we always define $L^p\loc(G,\cH)$ with respect to the Haar measure on $G$ and an increasing sequence of compact subsets $K_n \subset G$ such that the interiors of $K_n$ cover $G$. Note that two such increasing sequences are always cofinal, so that $L^p\loc(G,\cH)$ is unambiguously defined. Note that we have the continuous inclusion $C(G,\cH) \subset L^p\loc(G,\cH)$.

Also the following lemma is standard and we omit the proof.

\begin{lemma}\label{lem.make-acyclic}
Let $G$ be a lcsc group, $P$ an algebra and $\cK$ a Fr\'{e}chet $G$-$P$-bimodule. Define the Fr\'{e}chet $G$-$P$-bimodule $\cH := L^2\loc(G,\cK)$ with the left $G$-action and right $P$-action given by
$$(g \cdot \xi)(h) = g \cdot \xi(g^{-1} h) \quad\text{and}\quad (\xi \cdot a)(h) = \xi(h) \cdot a \; .$$
Then the complexes of Fr\'{e}chet $P$-modules given by
\begin{align*}
& 0 \longrightarrow \cH^G \longrightarrow \cH \pijl{d_0} C(G,\cH) \pijl{d_1} C(G^2,\cH) \pijl{d_2} \cdots \quad\text{and}\\
& 0 \longrightarrow \cH^G \longrightarrow \cH \pijl{d_0} L^2\loc(G,\cH) \pijl{d_1} L^2\loc(G^2,\cH) \pijl{d_2} \cdots \; ,
\end{align*}
with $d_n$ defined by \eqref{eq.bar-coboundary}, are strongly exact. In particular, $\cH$ is a strongly acyclic $G$-$P$-bimodule.
\end{lemma}

As a consequence, we recover the fact proven in \cite{Bl77} that cohomology for $G$ may be computed by using locally square integrable functions.

\begin{proposition}[{\cite{Bl77}}] \label{prop.bar-res-L2loc}
Let $G$ be a lcsc group, $P$ an algebra and $\cK$ a Fr\'{e}chet $G$-$P$-bimodule. Define a complex $\cC$ of Fr\'{e}chet $P$-modules given by
$$\cK \pijl{d_0} L^2\loc(G,\cK) \pijl{d_1} L^2\loc(G^2,\cK) \pijl{d_2} \cdots \; ,$$
with $d_n$ defined by \eqref{eq.bar-coboundary}. Then the inclusion maps $C(G^n,\cK) \recht L^2\loc(G^n,\cK)$ induce $P$-linear isomorphisms
$$H^n(G,\cK) \cong H^n(\cC) \quad\text{and}\quad \Hred^n(G,\cK) \cong \Hred^n(\cC) \; .$$
\end{proposition}
\begin{proof}
This is an immediate corollary of Proposition \ref{prop.resolution} and Lemma \ref{lem.make-acyclic}.
\end{proof}

Note that Proposition \ref{prop.bar-res-L2loc} has the following slightly unexpected consequence: if a continuous $n$-cocycle $\om : G^n \recht \cK$ is approximately inner in the $L^2\loc$-topology, then it must be approximately inner in the stronger topology of uniform convergence on compact sets.

\subsection{Change of coefficients~: a dimension formula}\label{subsec.coeff-change}

Fix a lcsc group $G$ and a von Neumann algebra $M$. A \emph{Hilbert $G$-$M$-bimodule} $\cH$ is a Hilbert space $\cH$ equipped with a strongly continuous unitary representation of $G$ and a normal antihomomorphism $M \recht B(\cH)$ so that the left $G$-action and the right $M$-action on $\cH$ commute.

Assume now that $\Tr$ is a normal semifinite faithful trace on $M$ and that $N \subset M$ is a von Neumann subalgebra such that $\Tr_{|N}$ is semifinite. Equip $N$ with the trace given by restricting $\Tr$ to $N$. Denote by $E : M \recht N$ the unique $\Tr$-preserving conditional expectation. Whenever $\cH$ is a right Hilbert $N$-module, the Connes tensor product $\cH \ovt_N L^2(M)$ is defined as the separation-completion of $\cH \otalg M$ with respect to the scalar product
$$\langle \xi \ot a, \eta \ot b \rangle = \langle \xi \cdot E(ab^*),\eta \rangle \; .$$
The formula $(\xi \ot a)\cdot b = \xi \ot ab$ turns $\cH \ovt_N L^2(M)$ into a right Hilbert $M$-module. If $\cH$ was a Hilbert $G$-$N$-bimodule, the formula $g \cdot(\xi \ot a) = (g \cdot \xi) \ot a$ turns $\cH \ovt_N L^2(M)$ into a Hilbert $G$-$M$-bimodule. The following proposition is analogous to \cite[Theorem 6.29]{Lu02}.

\begin{proposition}\label{prop.change-coeff}
Let $G$ be a lcsc group and $(M,\Tr)$ a von Neumann algebra with separable predual equipped with a normal semifinite faithful trace. Let $N \subset M$ be a von Neumann subalgebra such that $\Tr_{|N}$ is semifinite. Equip $N$ with the trace given by restricting $\Tr$ to $N$.

For every Hilbert $G$-$N$-module $\cH$, we have
\begin{align*}
& \dim_N H^n(G,\cH) \leq \dim_M H^n(G,\cH \ovt_N L^2(M)) \quad\text{and}\\
& \dim_N \Hred^n(G,\cH) = \dim_M \Hred^n(G,\cH \ovt_N L^2(M)) \; .
\end{align*}
\end{proposition}

In order to prove Proposition \ref{prop.change-coeff}, we need two elementary lemmas.

\begin{lemma}\label{lem.closed-subspace}
Let $(M,\Tr)$ be a von Neumann algebra equipped with a normal semifinite faithful trace. Let $N \subset M$ be a von Neumann subalgebra such that $\Tr_{|N}$ is semifinite. Assume that $\cH$ is a right Hilbert $N$-module and that $\cK \subset \cH$ is a closed $N$-submodule. Denote by $E : M \recht N$ the unique $\Tr$-preserving conditional expectation and  by
$$\cE : \cH \ovt_N L^2(M) \recht \cH : \cE(\xi \ot a) = \xi \cdot E(a)$$
the orthogonal projection of $\cH \ovt_N L^2(M)$ onto $\cH$.

Identifying $\cK \ovt_N L^2(M)$ with the closed $M$-submodule of $\cH \ovt_N L^2(M)$ given by the closed linear span of $\{\xi \ot a \mid \xi \in \cK, a \in M\}$, we have
$$\cK \ovt_N L^2(M) = \{\xi \in \cH \ovt_N L^2(M) \mid \cE(\xi \cdot a) \in \cK \;\;\text{for all}\;\; a \in M \;\} \; .$$
\end{lemma}
\begin{proof}
Denote by $P \colon \cH \recht \cK$ the orthogonal projection of $\cH$ onto $\cK$ and note that $P$ is $N$-linear. Put $Q = 1-P$. Then $P \ot_N 1$ is the orthogonal projection of $\cH \ovt_N L^2(M)$ onto $\cK \ovt_N L^2(M)$. Assume that $\xi \in \cH \ovt_N L^2(M)$ and that $\cE(\xi \cdot a) \in \cK$ for all $a \in M$. It follows that
$$\cE\bigl( (Q \ot_N 1)(\xi) \cdot a\bigr) = \cE\bigl( (Q \ot_N 1)(\xi \cdot a) \bigr) = Q(\cE(\xi \cdot a)) = 0$$
for all $a \in M$. Hence $(Q \ot_N 1)(\xi) = 0$, so that $\xi \in \cK \ovt_N L^2(M)$.
\end{proof}

\begin{lemma}\label{lem.general-coeff-change}
Let $(M,\Tr)$ be a von Neumann algebra with separable predual equipped with a normal semifinite faithful trace. Let $N \subset M$ be a von Neumann subalgebra such that $\Tr_{|N}$ is semifinite. Denote by $E : M \recht N$ the unique $\Tr$-preserving conditional expectation.

Let $\cH$ be an $N$-module and let $\cK$ be an $M$-module. Assume that $\theta : \cH \recht \cK$ and $\cE : \cK \recht \cH$ are $N$-linear maps satisfying
$$\cE(\theta(\xi) \cdot a) = \xi \cdot E(a) \;\;\text{for all}\;\; \xi \in \cH \; , \;  a \in M \; .$$
Then $\dim_N \cH \leq \dim_M \cK$.
\end{lemma}
\begin{proof}
We first prove the lemma when $\Tr$ is a tracial state. Assume that $p \in \M_k(\C) \ot N$ is a projection and
$$\vphi : p (\M_{k,1}(\C) \ot N) \recht \cH$$
is an injective $N$-linear map. We construct an injective $M$-linear map $p (\M_{k,1}(\C) \ot M) \recht \cK$. The inequality $\dim_N \cH \leq \dim_M \cK$ then follows directly from \eqref{eq.def-dim}.

Define $\xi \in \M_{1,k}(\C) \ot \cH$ given by
$$\xi := \sum_{i=1}^k e_{1i} \ot \vphi(p(e_{i1} \ot 1)) \; .$$
A direct computation yields that $\xi p = \xi$ and that $\vphi(\eta) = \xi \eta$ for all $\eta \in p(\M_{k,1}(\C) \ot N)$. Define
$$\psi : p (\M_{k,1}(\C) \ot M) \recht \cK : \psi(\eta) = (\id \ot \theta)(\xi) \eta \; .$$
By construction, $\psi$ is $M$-linear. We claim that $\psi$ is injective. So assume that $\eta \in p(\M_{k,1}(\C) \ot M)$ and that $\psi(\eta) = 0$. Then also
$$0 = \psi(\eta) \eta^* = (\id \ot \theta)(\xi) \, \eta \eta^* \; .$$
Applying $\id \ot \cE$, it follows that $\xi \, (\id \ot E)(\eta \eta^*) = 0$. Since $\vphi$ is injective and $(\id \ot E)(\eta \eta^*)$ belongs to $p (\M_k(\C) \ot N) p$, we get that $(\id \ot E)(\eta \eta^*) = 0$. Since $E$ is faithful, we conclude that $\eta = 0$. So $\psi$ is injective and the lemma is proven in the case where $\Tr$ is a finite trace.

In the general case, choose an increasing sequence of projections $p_n \in N$ with $\Tr(p_n) < \infty$ for all $n$ and with the central support of $p_n$ in $N$ converging to $1$ strongly. Applying the previous case to the $p_n N p_n$-module $\cH p_n$ and the $p_n M p_n$-module $\cK p_n$, we conclude that
$$\Tr(p_n) \, \dim_{p_n N p_n} (\cH p_n) \leq \Tr(p_n) \, \dim_{p_n M p_n}(\cK p_n) \leq \dim_M \cK$$
for all $n$. Taking the limit $n \recht \infty$ and using Lemma \ref{lem.realize-semifinite-dim}, the lemma follows.
\end{proof}

We are now ready to prove Proposition \ref{prop.change-coeff}.

\begin{proof}[Proof of Proposition \ref{prop.change-coeff}]
Denote $\cK := \cH \ovt_N L^2(M)$ and denote by $E \colon  M \recht N$ the unique $\Tr$-preserving conditional expectation. The map
$$\theta : \cH \recht \cK : \theta(\xi) = \xi \ot 1$$
is a $G$-$N$-linear isometry with adjoint
$$\cE : \cK \recht \cH : \cE(\xi \ot a) = \xi \cdot E(a) \; .$$
The maps $\theta$ and $\cE$ naturally induce $N$-linear maps
$$\theta : H^n(G,\cH) \recht H^n(G,\cK) \quad\text{and}\quad \cE : H^n(G,\cK) \recht H^n(G,\cH)$$
satisfying the assumptions of Lemma \ref{lem.general-coeff-change}. So it follows from Lemma \ref{lem.general-coeff-change} that
$$\dim_N H^n(G,\cH) \leq \dim_M H^n(G,\cK) \; .$$

It remains to consider the reduced cohomologies. We first make the following observation~: whenever $\cL$ is a Hilbert $N$-module, we have
\begin{equation}\label{eq.ok-hilbert}
\dim_N \cL = \dim_M (\cL \ovt_N L^2(M)) \; .
\end{equation}
Indeed, we can write $\cL = p(\ell^2(\N) \ovt L^2(N))$ for some projection $p \in \B(\ell^2(\N)) \ovt N$. Then $\cL \ovt_N L^2(M) = p(\ell^2(\N) \ovt L^2(M))$ so that both $\dim_{N} \cL$ and $\dim_M(\cL \ovt_N L^2(M))$ are given by $(\Tr_{\B(\ell^2(\N))} \ot \Tr)(p)$, where $\Tr_{\B(\ell^2(\N))}$ is the canonical trace on $\B(\ell^2(\N))$.

Using Lemma \ref{lem.closed-subspace}, we get that $\cK^G = \cH^G \ovt_N L^2(M)$. In combination with \eqref{eq.ok-hilbert}, we get that
$$\dim_N \Hred^0(G,\cH) = \dim_M \Hred^0(G,\cK) \; .$$
To prove the same formula for the reduced $n$-cohomology, $n \geq 1$, consider the complexes
\begin{align*}
& \cH \pijl{d_0} L^2\loc(G,\cH) \pijl{d_1} L^2\loc(G^2,\cH) \pijl{d_2} \cdots \quad\text{and}\\
& \cK \pijl{d_0} L^2\loc(G,\cK) \pijl{d_1} L^2\loc(G^2,\cK) \pijl{d_2} \cdots \;\; ,
\end{align*}
where $d_n$ is defined by \eqref{eq.bar-coboundary}. Fix $n \geq 1$ and define
\begin{align*}
& Z^n(\cH) := \{\xi \in L^2\loc(G^n,\cH) \mid d_n(\xi) = 0 \} \quad\text{and}\\
& B^n(\cH) := \{d_{n-1}(\xi) \mid \xi \in L^2\loc(G^{n-1},\cH)\} \; .
\end{align*}
We similarly define $Z^n(\cK)$ and $B^n(\cK)$.
By Proposition \ref{prop.bar-res-L2loc}, we have
$$\dim_N \Hred^n(G,\cH) = \dim_N \Bigl(\frac{Z^n(\cH)}{\clos(B^n(\cH))}\Bigr) \quad\text{and}\quad
\dim_M \Hred^n(G,\cK) = \dim_{M} \Bigl(\frac{Z^n(\cK)}{\clos(B^n(\cK))}\Bigr) \; .$$

Fix an increasing sequence of compact subsets $K_k \subset G$ whose interiors cover $G$. Denote by $\vphi_k : L^2\loc(G^n,\cH) \recht L^2(K_k^n,\cH)$ the restriction map. We define $Z^n_k(\cH)$ as the closure of $\vphi_k(Z^n(\cH))$, and we define $B^n_k(\cH)$ as the closure of $\vphi_k(B^n(\cH))$. By construction, we have $N$-linear maps
$$\vphi_k : \frac{Z^n(\cH)}{\clos(B^n(\cH))} \recht \frac{Z^n_k(\cH)}{B^n_k(\cH)}$$
with dense range. Also by construction, $\Ker \vphi_k$ is a decreasing sequence of $N$-submodules with trivial intersection. It then follows from Lemma \ref{lem.weak-inverse-limit-semifinite} that
\begin{align}\label{dim-red-cohom-as-limit}
\dim_N \Hred^n(G,\cH) = \lim_k \dim_N\Bigl(\frac{Z^n_k(\cH)}{B^n_k(\cH)}\Bigr) \; .
\end{align}
We similarly have that
$$\dim_M \Hred^n(G,\cK) = \lim_k \dim_M\Bigl(\frac{Z^n_k(\cK)}{B^n_k(\cK)}\Bigr) \; .$$
To conclude the proof of the proposition, we identify the Hilbert $M$-modules
\begin{equation}\label{eq.aim}
\frac{Z^n_k(\cK)}{B^n_k(\cK)} \cong \frac{Z^n_k(\cH)}{B^n_k(\cH)} \ovt_N L^2(M) \; .
\end{equation}
Once \eqref{eq.aim} is proven, the proposition follows by using \eqref{eq.ok-hilbert}.

To prove \eqref{eq.aim}, note that $L^2(K_k^n,\cH) = L^2(K_k^n) \ovt \cH$, so that we can identify
$$L^2(K_k^n,\cH) \ovt_N L^2(M) = L^2(K_k^n,\cK) \; .$$
We therefore get inclusions
\begin{equation}\label{eq.inclusion}
Z^n_k(\cH) \ovt_N L^2(M) \subset Z^n_k(\cK) \quad\text{and}\quad B^n_k(\cH) \ovt_N L^2(M) \subset B^n_k(\cK)
\end{equation}
and it remains to prove that these inclusions are actually equalities. To prove this, we use Lemma \ref{lem.closed-subspace} and denote by $\cE \colon L^2(K_k^n,\cK) \recht L^2(K_k^n,\cH)$ the orthogonal projection. Fix $\xi \in Z^n_k(\cK)$ and fix $a \in M$. We must show that $\cE(\xi \cdot a) \in Z^n_k(\cH)$. The definition of $Z^n_k(\cK)$ provides a sequence $\om_i \in Z^n(G,\cK)$ such that $\xi = \lim_i \vphi_k(\om_i)$. We also have $\cE : Z^n(G,\cK) \recht Z^n(G,\cH)$ and get that
$$\cE(\xi \cdot a) = \lim_i \vphi_k( \cE(\om_i \cdot a)) \; .$$
Since $\cE(\om_i \cdot a)$ is a sequence in $Z^n(G,\cH)$, we indeed get that $\cE(\xi \cdot a) \in Z^n_k(\cH)$. This proves that the first inclusion in \eqref{eq.inclusion} actually is an equality. We similarly get that the second inclusion in \eqref{eq.inclusion} is an equality. The required identification \eqref{eq.aim} follows and the proposition is proven.
\end{proof}

As a consequence of the above proof we obtain the following result, which also appears in \cite[Proposition 3.8]{Pe11}.
\begin{porism}\label{vanishing-porism}
For any lcsc unimodular group $G$ and any $n\geq 0$ we have $\betar^{n}_{(2)}(G)=0$ if and only if $\underline{H}^n(G,L^2(G))$ vanishes.
\end{porism}
\begin{proof}
Choosing $\cH:=L^2(G)$ in Proposition \ref{prop.change-coeff} and its proof, we see that the modules $\frac{Z^n_k(\cH)}{B^n_k(\cH)}$ appearing in \eqref{dim-red-cohom-as-limit} are actually Hilbert $L(G)$-modules. Since the dimension function is faithful on the class of Hilbert $LG$-modules, if we assume that $\betar^{n}_{(2)}(G)=0$ this forces $\frac{Z^n_k(\cH)}{B^n_k(\cH)}=\{0\}$ for each $k\geq 0$. At the same time, the kernels of the maps $\varphi_k\colon \underline{H}^n(G,L^2(G))\to \frac{Z^n_k(\cH)}{B^n_k(\cH)}$ have trivial intersection and hence $\underline{H}^n(G,L^2(G))=\{0\}$.
\end{proof}

\section{Cohomology of countable equivalence relations}\label{sec.cohom-equiv-rel}

Fix a countable Borel pmp equivalence relation $\cR$ on the standard probability space $(Y,\nu)$. Denote by $[[\cR]]$ the full pseudogroup of $\cR$, i.e.\ the set of all partial Borel bijections $\psi$ with domain $D(\psi) \subset Y$ and range $R(\psi) \subset Y$, such that $(y,\psi(y)) \in \cR$ for all $y \in D(\psi)$.
We write $\cR^{(0)} := Y$ and $\cR^{(n)} := \{(y_0,\ldots,y_n) \mid (y_i,y_j) \in \cR \;\text{for all}\; i,j \}$. All $\cR^{(n)}$ are equipped with the natural $\sigma$-finite measure $\nu^{(n)}$, with $\nu^{(0)} = \nu$ and with $\nu^{(n)}$ given by integrating w.r.t.\ $\nu$ the counting measure over the projection $\cR^{(n)} \recht Y$ onto any of the coordinates.

The von Neumann algebra $L \cR$ of the equivalence relation $\cR$ is defined as the von Neumann algebra acting on $L^2(\cR^{(1)},\nu^{(1)})$ generated by the partial isometries $u_\vphi$, $\vphi \in [[\cR]]$, given by
$$(u_\vphi \cdot \xi)(y,z) = \begin{cases} \xi(\vphi^{-1}(y),z) &\quad\text{if $y \in R(\vphi)$,}\\ 0 &\quad\text{otherwise}.\end{cases}$$
The unit vector $\chi \in L^2(\cR^{(1)},\nu^{(1)})$ given by $\chi(y,z) = 1$ if $y=z$ and $\chi(y,z) = 0$ if $y \neq z$ implements a faithful normal tracial state $\tau$ on $L \cR$ satisfying
$$\tau(u_\vphi) = \nu\bigl(\{x \in D(\vphi) \mid \vphi(x) = x\}\bigr)$$
for all $\vphi \in [[\cR]]$. We refer to \cite{FM75} for the details of the construction of $L \cR$.

We can identify $L^2(L \cR, \tau)$ with $L^2(\cR,\nu^{(1)})$ and under this identification, the right action of $L \cR$ on $L^2(\cR,\nu^{(1)})$ is given by
$$(\xi \cdot u_\vphi)(y,z) = \begin{cases} \xi(y,\vphi(z)) &\quad\text{if $z \in D(\vphi)$,}\\ 0 &\quad\text{otherwise}.\end{cases}$$

For later use, we write in this section a concrete complex of Fr\'{e}chet $L \cR$-modules such that the $L \cR$-dimensions of the cohomology modules precisely are the $L^2$-Betti numbers of $\cR$,  as defined in \cite{Ga01}.

Fix a countable subset $\Lambda \subset [[\cR]]$ with $\id \in \Lambda$ and such that
$$\cR = \bigcup_{\vphi \in \Lambda} \graph(\vphi) \; .$$
Enumerate $\Lambda = \bigcup_k \Lambda_k$ as an increasing sequence of finite subsets with $\id \in \Lambda_0$,  and define
\begin{equation}\label{eq.sets-Sigma}
\Sigma^{(n)} := \{(y_0,\ldots,y_n,z) \in \cR^{(n+1)} \mid y_i \neq y_j \;\;\text{whenever}\;\; i \neq j \} \; .
\end{equation}
We then consider the increasing sequence of subsets $\Sigma^{(n)}_k \subset \Sigma^{(n)}$ given by
\begin{equation}\label{eq.subsets-Sigma}
\begin{split}
\Sigma^{(n)}_k := \{(y_0,\ldots,y_n,z) \in \Sigma^{(n)} \mid \; & \exists \vphi_0,\ldots,\vphi_n \in \Lambda_k , \exists y \in D(\vphi_0) \cap \cdots \cap D(\vphi_n)\\ &\text{such that}\;\; y_i = \vphi_i(y) \;\;\text{for all}\;\;i=0,\ldots,n \} \; ,
\end{split}
\end{equation}
and equip $\Sigma^{(n)}$ with the $\sigma$-finite measure given by restricting $\nu^{(n+1)}$.
We consider the Fr\'{e}chet spaces $L^2\locS(\Sigma^{(n)})$ where we use the notation $L^2\locS$ to stress that we take functions that are square integrable on all the subsets $\Sigma^{(n)}_k$. Note that for $n = 0$, we just obtain $L^2(\cR)$, because $\Sigma^{(0)} = \Sigma^{(0)}_k = \cR$ for all $k$.

Every $L^2(\Sigma^{(n)}_k)$ is a right Hilbert $L \cR$-module under
$$(\xi \cdot u_\vphi)(y_0,\ldots,y_n,z) = \begin{cases} \xi(y_0,\ldots,y_n,\vphi(z)) &\quad\text{if $z \in D(\vphi)$,} \\ 0 &\quad\text{otherwise}.\end{cases}$$
In this way, $L^2\locS(\Sigma^{(n)})$ becomes a right Fr\'{e}chet $L \cR$-module.

We denote by $\beta_n^{(2)}(\cR)$ the $L^2$-Betti numbers of the equivalence relation $\cR$, as defined in \cite[D\'{e}finition 3.14]{Ga01}.

\begin{proposition}\label{prop.compute-Betti-R}
Consider the complex $\cC$ of right Fr\'{e}chet $L \cR$-modules given by
$$L^2\locS(\Sigma^{(0)}) \pijl{d_0} L^2\locS(\Sigma^{(1)}) \pijl{d_1} L^2\locS(\Sigma^{(2)}) \pijl{d_2} \cdots$$
where $d_n$ is given by
$$(d_n \om)(y_0,\ldots,y_{n+1},z) = \sum_{i=0}^{n+1} (-1)^i \om(y_0,\ldots,\widehat{y_i},\ldots,y_{n+1}, z) \; .$$
Then
$$\beta_n^{(2)}(\cR) = \dim_{L \cR} H^n(\cC) = \dim_{L \cR} \Hred^n(\cC) \; .$$
\end{proposition}
\begin{proof}
Using the same formulae as for $d_n$, we also have, for every $k \in \N$, the complexes of \emph{finitely generated Hilbert} $L \cR$-modules given by
$$L^2(\Sigma_k^{(0)}) \pijl{d^0_k} L^2(\Sigma_k^{(1)}) \pijl{d^1_k} L^2(\Sigma_k^{(2)}) \pijl{d^2_k} \cdots \; .$$
We write $C^n := L^2\locS(\Sigma^{(n)})$ and $C^n_k := L^2(\Sigma^{(n)}_k)$.
By definition of $L^2\locS$, the Fr\'{e}chet $L \cR$-module $C^n$ is the inverse limit of the Hilbert $L \cR$-modules $C^n_k$.
Denote
$$Z^n := \Ker(C^n \pijl{d_n} C^{n+1}) \quad\text{and}\quad Z^n_k := \Ker(C^n_k \pijl{d_n} C^{n+1}_k) \; .$$
By construction, the Fr\'{e}chet $L \cR$-module $Z^n$ is the inverse limit of the finitely generated Hilbert $L \cR$-modules $Z^n_k$. By Proposition \ref{prop.form-inv-limit}, we have
\begin{equation}\label{eq.dim-equalities}
\begin{split}
\dim_{L \cR} \frac{Z^n}{d_{n-1}(C^{n-1})} &= \dim_{L \cR} \frac{Z^n}{\clos\bigl(d_{n-1}(C^{n-1})\bigr)}
\\ &= \dim_{L \cR} \Bigl( \invlimit \frac{Z^n_k}{\clos\bigl(d^{n-1}_k(C^{n-1}_k)\bigr)} \Bigr) \; .
\end{split}
\end{equation}
For every bounded operator $T$ between two Hilbert spaces, we denote by $\oIm T$ the closure of the image of $T$.
So, using Proposition \ref{prop.inv-limit-CG}, we get that
\begin{align}
\dim_{L \cR} H^n(\cC) &= \dim_{L \cR} \Hred^n(\cC) \notag\\
&= \lim_\alpha \Bigl(\lim_\beta \Bigl( \dim_{L \cR} \oIm \Bigl( \frac{Z^n_\beta}{\clos (d^{n-1}_\beta(C^{n-1}_\beta))} \recht \frac{Z^n_\alpha}{\clos(d^{n-1}_\alpha(C^{n-1}_\alpha))} \Bigr) \Bigr) \Bigr) \; .\label{eq.nog-dim-eq}
\end{align}
It remains to prove that the expression in \eqref{eq.nog-dim-eq} equals $\beta_n^{(2)}(\cR)$.

Whenever $K$ is a closed subspace of a Hilbert space, denote by $P_K$ the orthogonal projection onto $K$. Denote by $\pi_{\alpha,\beta} : C^n_\beta \recht C^n_\alpha$ the $L \cR$-linear operator given by restricting functions on $\Sigma^{(n)}_\beta$ to $\Sigma^{(n)}_\alpha$.
We can then identify
\begin{align}
& \oIm \Bigl( \frac{Z^n_\beta}{\clos(d^{n-1}_\beta(C^{n-1}_\beta))} \recht \frac{Z^n_\alpha}{\clos(d^{n-1}_\alpha(C^{n-1}_\alpha))} \Bigr)
\qquad\text{with}\notag\\
& \oIm \Bigl((P_{Z^n_\alpha} - P_{\clos(d^{n-1}_\alpha(C^{n-1}_\alpha))}) \circ \pi_{\alpha,\beta} \circ
(P_{Z^n_\beta} - P_{\clos(d^{n-1}_\beta(C^{n-1}_\beta))})\Bigr) \; .\label{eq.almost}
\end{align}
For every $L \cR$-linear operator $T$ between Hilbert $L \cR$-modules, we know that
$$\dim_{L \cR} \oIm T = \dim_{L \cR} \oIm T^* \; .$$
Therefore, the $L \cR$-dimension of \eqref{eq.almost} equals the $L \cR$-dimension of
\begin{equation}\label{eq.adjoint}
\oIm \Bigl((P_{Z^n_\beta} - P_{\clos(d^{n-1}_\beta(C^{n-1}_\beta))}) \circ \pi_{\alpha,\beta}^* \circ
(P_{Z^n_\alpha} - P_{\clos(d^{n-1}_\alpha(C^{n-1}_\alpha))}) \Bigr) \; .
\end{equation}
Since $\Hred^p(\cC_i,d_i )=\Hred_p(\cC_i,d_i^*)$ for any Hilbert chain complex $(\cC_i,d_i)$, this can in turn be identified with
\begin{equation}\label{eq.homology}
\oIm \Bigl( \frac{\Ker (d_\alpha^{n-1})^*}{\clos(\image (d_\alpha^n)^*)} \pijl{\pi_{\alpha,\beta}^*} \frac{\Ker (d_\alpha^{n-1})^*}{\clos(\image (d_\alpha^n)^*)}\Bigr) \; .
\end{equation}
Denote by $\nabla_n(\alpha,\beta)$ the $L \cR$-dimension of the Hilbert $L \cR$-module in \eqref{eq.homology}. We have shown that
\begin{equation}\label{eq.joepie}
\dim_{L \cR} H^n(\cC) = \dim_{L \cR} \Hred^n(\cC) = \lim_\alpha \bigl(\lim_\beta \nabla_n(\alpha,\beta)\bigr) \; .
\end{equation}

Equipped with the projection $\pi : \Sigma^{(n)} \recht Y : \pi(y_0,\ldots,y_n,z) = z$ and the action of $\cR$ on the last variable of $\Sigma^{(n)}$, we get that $\Sigma$ is an $\cR$-simplicial complex in the sense of \cite[D\'{e}finition 2.6]{Ga01}. Since $\Sigma$ is $n$-connected for all $n$, it follows from
\cite[D\'{e}finition 3.14]{Ga01} that $\beta_n^{(2)}(\cR)$ equals the $n$-th $L^2$-Betti number the $\cR$-simplicial complex $\Sigma$. Now the sets $\Sigma_\alpha^{(n)}$ define an increasing sequence of subcomplexes $\Sigma_\alpha \subset \Sigma$. The subcomplexes $\Sigma_\alpha$ are uniformly locally bounded (in the sense of \cite[D\'{e}finition 2.7]{Ga01}) and the space of $L^2$-$n$-chains of $\Sigma_\alpha$ is exactly $C^n_\alpha$. The boundary operators are exactly the operators $(d^n_\alpha)^*$. So it follows from \cite[Proposition 3.9]{Ga01} that
$$\beta_n^{(2)}(\cR) = \lim_\alpha \bigl(\lim_\beta \nabla_n(\alpha,\beta)\bigr) \; .$$
Together with \eqref{eq.joepie}, the proposition is proven.
\end{proof}

\section{\boldmath $L^2$-Betti numbers for locally compact groups and their cross section equivalence relations}

\subsection{Cocompact cross sections and their equivalence relations}\label{subsec.cross-section}

\begin{definition}
Let $G$ be a lcsc group, $(X,\mu)$ a standard probability space and $G \actson (X,\mu)$ an essentially free nonsingular action.
\begin{itemize}
\item We call a Borel set $Y \subset X$ a \emph{cross section} of $G \actson (X,\mu)$ if there exists a neighborhood of the identity $\cU \subset G$ such that the map $\cU \times Y \recht X : (g,y) \mapsto g \cdot y$ is injective and such that $\mu(X - G \cdot Y) = 0$.
\item We call the cross section $Y \subset X$ \emph{cocompact} if there exists a compact subset $K \subset G$ such that $K \cdot Y$ is a $G$-invariant Borel subset of $X$ and $\mu(X - K \cdot Y) = 0$.
\end{itemize}
\end{definition}

Note that the injectivity of $\cU \times Y \recht X$ implies that the map $G \times Y \recht X : (g,y) \mapsto g \cdot y$ is countable-to-one and hence maps Borel sets to Borel sets.

The following theorem was proven in \cite[Proposition 2.10]{Fo74}, although the cocompactness was not studied there. Since it is crucial for us to have cocompact cross sections, we give a detailed proof.

\begin{theorem}[{\cite[Proposition 2.10]{Fo74}}] \label{thm.cross-section}
Every essentially free nonsingular action of a lcsc group $G$ on a standard probability space admits a cocompact cross section.
\end{theorem}
\begin{proof}
Fix a lcsc group $G$, a standard probability space $(X,\mu)$ and an essentially free nonsingular action $G \actson (X,\mu)$. Fix a compact neighborhood $K_0$ of $e$ in $G$ and put $K_1 := K_0^{-1} K_0$. We start by proving the following claim.

{\bf Step 1.} {\it If $\cW \subset X$ is a nonnegligible Borel subset of $X$, there exists a Borel subset $Y \subset \cW$ such that the map $K_0 \times Y \recht X : (k,y) \mapsto k \cdot y$ is injective and has nonnegligible image.}

{\bf Proof of step 1.} By \cite[Theorem 3.2]{Va62}, there exists a compact metric space $(P,d)$ and a continuous action $G \actson P$ by homeomorphisms such that we can view $X$ as a $G$-invariant Borel subset of $P$. We extend $\mu$ to a measure on $P$ by putting $\mu(P-X) = 0$. Put $L := K_1 K_1 - \operatorname{interior}(K_1)$. Then $L$ is compact and $e \not\in L$. Since $G$ acts continuously on $P$ and since $L$ is compact, we can define the continuous function
$$\delta : P \recht [0,+\infty) : \delta(x) = \min \{ d(k \cdot x,x) \mid k \in L\} \; .$$
Since $G \actson (X,\mu)$ is essentially free, we see that $\delta(x) > 0$ for a.e.\ $x \in X$. So we can take $\eps > 0$ such that the set
$$\cW_1 := \{x \in \cW \mid \delta(x) \geq 2 \eps \}$$
satisfies $\mu(\cW_1) > 0$. Take a closed ball $B \subset P$ with diameter smaller than $\eps$ such that $\mu(B \cap \cW_1) > 0$. Denote by $X_0$ the conegligible $G$-invariant Borel set of all $x \in X$ that have trivial stabilizer. So also $\mu(B \cap \cW_1 \cap X_0) > 0$.
By regularity of $\mu$ (see e.g.\ \cite[Theorem 17.10]{Ke95}), take a compact subset $P_1 \subset B \cap \cW_1 \cap X_0$ with $\mu(P_1) > 0$. Define the compact set $\cS \subset P_1 \times P_1$ given by
$$\cS := \{(x,y) \in P_1 \times P_1 \mid \exists k \in K_1, y = k \cdot x \} \; .$$
It is clear that $(x,x) \in \cS$ for all $x \in P_1$ and that $(y,x) \in \cS$ if and only if $(x,y) \in \cS$, because $K_1 = K_1^{-1}$. But $\cS$ is also transitive: if $(x,y) \in \cS$ and $(y,z) \in \cS$, then $(x,z) \in \cS$. Indeed, take $r,s \in K_1$ such that $y = r \cdot x$ and $z = s \cdot y$. Then $sr \in K_1 K_1$ and $(sr) \cdot x = z$. Since $x$ and $z$ both belong to the ball $B$ with diameter $\eps$, we get that $d((sr) \cdot x, x) \leq \eps$. Since $x \in \cW_1$, we know that $\delta(x) \geq 2 \eps$. So we must have that $sr \in \operatorname{interior}(K_1)$, and hence $(x,z) \in \cS$.

It follows that $\cS$ is an equivalence relation on the compact metric space $P_1$. The $\cS$-orbit of $x \in P_1$ is given by $K_1 \cdot x \cap P_1$ and hence, $\cS$ has closed orbits. So by \cite[Theorem A.15]{Ta79}, $\cS$ admits a fundamental domain: we can choose a Borel subset $Y \subset P_1$ that meets every $\cS$-orbit exactly once.
By construction, we have $Y \subset P_1 \subset \cW \cap X_0$. We can see as follows that $Y$ satisfies all the conditions in the claim.
\begin{itemize}
\item The map $K_0 \times Y \recht X : (k,y) \mapsto k \cdot y$ is injective. Indeed, if $k \cdot y = s \cdot z$ for $k,s \in K_0$ and $y,z \in Y$, we get that $s^{-1} k \in K_1$ and $(s^{-1} k) \cdot y = z$. It follows that $(y,z) \in \cS$ and hence $y = z$, because $y$ and $z$ belong to the fundamental domain $Y$ of $\cS$. Since $y$ has trivial stabilizer, also $s = k$.
\item Since the map $K_0 \times Y \recht X :(k,y) \mapsto k \cdot y$ is Borel and injective, $K_0 \cdot Y$ is a Borel subset of $X$.
Since $K_0$ has a nonempty interior, we can write $G = \bigcup_n g_n K_0$ for a sequence of group elements $g_n \in G$. Then $G \cdot Y = \bigcup_n g_n \cdot (K_0 \cdot Y)$. Since $P_1 \subset G \cdot Y$, the Borel set $G \cdot Y$ is nonnegligible. Since the action $G \actson (X,\mu)$ is nonsingular, it follows that also $K_0 \cdot Y$ is nonnegligible.
\end{itemize}
This proves step 1.

{\bf Step 2.} {\it There exists a Borel set $Z \subset X$ such that the map $K_0 \times Z \recht X : (k,y) \mapsto k \cdot y$ is injective and such that $K_1 \cdot Z$ has complement of measure zero.}

{\bf Proof of step 2.} Take a maximal family of disjoint nonnegligible Borel subsets $\cW_n \subset X$ that can be written as $\cW_n = K_0 \cdot Z_n$ for some Borel set $Z_n \subset X$ and with the map $K_0 \times Z_n \recht X : (k,y) \mapsto k \cdot y$ being injective. Since $\mu$ is a probability measure, this family $(\cW_n )$ is countable. Put $Z = \bigcup_n Z_n$. Then $Z$ is a Borel set and since the sets $\cW_n$ are disjoint, the map $K_0 \times Z \recht X : (k,y) \mapsto k\cdot y$ is injective. We claim that $X - K_1 \cdot Z$ has measure zero. If not, step~1 provides us with a Borel subset $Y \subset X - K_1 \cdot Z$ such that the map $K_0 \times Y \recht X$ is injective and has nonnegligible image. Since $Y \cap K_1 \cdot Z = \emptyset$, also $K_0 \cdot Y \cap K_0 \cdot Z = \emptyset$. So we could add the nonnegligible set $K_0 \cdot Y$ to the family $(\cW_n)$, contradicting its maximality. This ends the proof of step~2.

{\bf End of the proof of Theorem \ref{thm.cross-section}.} Since $G \actson (X,\mu)$ is essentially free, we start by discarding a $G$-invariant Borel set of measure zero so that $G \actson X$ becomes a free action. By step~2, take a Borel set $Z \subset X$ such that the map $K_0 \times Z \recht X : (k,y) \mapsto k \cdot y$ is injective and such that $K_1 \cdot Z$ has complement of measure zero. Put $\cW := K_1 \cdot Z$. Since $K_0$ has a nonempty interior, we can choose a sequence $g_n \in G$ such that $G = \bigcup_n K_0 g_n$. Put $A = \bigcap_n g_n^{-1} \cdot \cW$. Then $A$ is a Borel set and $\mu(X-A) = 0$. By \cite[Lemma B.8]{Zi84}, we can choose a Borel set $B \subset A$ such that $\mu(A - B) = 0$ and such that $X_0 := G \cdot B$ is a Borel set. Since $B \subset X_0$, we have $\mu(X - X_0) = 0$. For every $n \in \N$, we have
$$K_0 g_n \cdot B \subset K_0 g_n \cdot A \subset K_0 \cdot \cW = K_0 K_1 \cdot Z \; .$$
Putting $K := K_0 K_1$ and taking the union over $n$, we get that $X_0 \subset K \cdot Z$. Since $X_0$ is $G$-invariant, this means that $X_0 = K \cdot (Z \cap X_0)$.

We define $Y := Z \cap X_0$. We have proven that the map $K_0 \times Y \recht X : (k,y) \mapsto k \cdot y$ is injective and that $K \cdot Y = X_0$ is a $G$-invariant Borel set with complement of measure zero. So $Y$ is a cocompact cross section for $G \actson (X,\mu)$.
\end{proof}

The following proposition contains the basic properties of the cross section equivalence relation. The results are well known but not explicitly stated in the literature, so for the convenience of the reader we include a proof in Appendix~\ref{appB}.

\begin{proposition}\label{prop.properties-cross-section-eq-rel}
Let $G$ be a lcsc unimodular group and $G \actson (X,\mu)$ an essentially free pmp action on a standard probability space. Let $Y \subset X$ be a cross section and fix a Haar measure $\lambda$ on $G$.
\begin{enumerate}
\item\label{p1} The formula $\cR := \{(y,y') \in Y \times Y \mid y \in G \cdot y'\}$ defines a countable Borel equivalence relation on $Y$.

\item\label{p2} The set $Z := \{(x,y) \in X \times Y \mid x \in G \cdot y\}$ is Borel. The projection on the first coordinate $\pi_\ell : Z \recht X$ is countable-to-one. Define the measure $\eta$ on $Z$ by integrating w.r.t.\ $\mu$ the counting measure over the map $\pi_\ell$.

    There exist a unique probability measure $\nu$ on $Y$ and a unique $0 < \covol Y < +\infty$ such that
    \begin{equation}\label{eq.first-push}
    \Psi_*(\lambda \times \nu) = \covol Y \cdot \eta \quad\text{where}\quad \Psi : G \times Y \recht Z : \Psi(g,y) = (g \cdot y, y) \; .
    \end{equation}
    In particular, whenever $\cU$ is a neighborhood of $e$ in $G$ such that $\theta: \cU \times Y \recht X : (g,y) \mapsto g \cdot y$ is injective, we have
    \begin{equation}\label{eq.another-push}
\theta_*(\lambda_{|\cU} \times \nu) = \covol Y \cdot \mu_{|\cU \cdot Y} \; .
\end{equation}

\item\label{p3} The probability measure $\nu$ is $\cR$-invariant.

\item\label{p4} If $Y' \subset X$ is a different cross section with corresponding equivalence relation $\cR'$, then

there exist Borel subsets $Y_0 \subset Y$ and $Y_0' \subset Y'$ and a Borel bijection $\al : Y_0 \recht Y_0'$ satisfying the following properties.
\begin{itemize}
\item $Y_0$ meets a.e.\ $\cR$-orbit and $Y_0'$ meets a.e.\ $\cR'$-orbit.
\item We have $\displaystyle \al_*(\nu_{|Y_0}) = \frac{\covol Y}{\covol Y'} \, \nu'_{|Y_0'} \; .$
\item $\al$ is an isomorphism between the restricted equivalence relations $\cR_{|Y_0}$ and $\cR'_{|Y_0'}$.
\end{itemize}
In particular, when $G \actson (X,\mu)$ is ergodic, the equivalence relations $\cR$ and $\cR'$ are stably orbit equivalent with compression constant $\covol(Y) / \covol(Y')$.

\item\label{p5} $(\cR,\nu)$ is ergodic if and only if $G \actson (X,\mu)$ is ergodic.

\item\label{p6} $(\cR,\nu)$ has infinite orbits a.e.\ if and only if $G$ is noncompact.

\item\label{p7} $(\cR,\nu)$ is amenable if and only if $G$ is amenable.
\end{enumerate}
\end{proposition}

We end this section with a simple lemma which will be needed in the proof of Theorem \ref{thm.mainA}.

\begin{lemma}\label{lem.compact-finite}
Let $G$ be a lcsc unimodular group and $G \actson (X,\mu)$ an essentially free pmp action. Let $Y \subset X$ be a cross section and $\cR$ the cross section equivalence relation on $Y$ as in Proposition \ref{prop.properties-cross-section-eq-rel}. Denote by $[[\cR]]$ the full pseudogroup of $\cR$, i.e.\ the set of all partial Borel bijections of $Y$ that have their graph in $\cR$.
For every compact subset $C \subset G$, there exists a finite subset $\cF \subset [[\cR]]$ such that for all $y \in Y$, we have $Y \cap (C\cdot y) = \cF \cdot y$. Here we use the notation $\cF \cdot y := \{\vphi(y) \mid \vphi \in \cF, y \in D(\vphi)\}$.
\end{lemma}

\begin{proof}
Define $\cS := \{(y,z) \in Y \times Y \mid z \in C \cdot y\}$. Denote by $\pi_\ell : \cR \recht Y$ and $\pi_r : \cR \recht Y$ the projections on the first and second coordinate. To prove the lemma, it suffices to show that there exists a $\kappa > 0$ such that
$$\# (\cS \cap \pi_\ell^{-1}(\{y\})) \leq \kappa \quad\text{and}\quad \# (\cS \cap \pi_r^{-1}(\{z\})) \leq \kappa \quad\text{for all}\;\; y,z \in Y \; .$$
Since $\cS$ can also be written as $\{(y,z) \in Y \times Y \mid y \in C^{-1} \cdot z\}$, it suffices to prove the first inequality.

Take a neighborhood $\cU$ of $e$ in $G$ such that $\theta : \cU \times Y \recht X : (g,y) \mapsto g \cdot y$ is injective. Take $\kappa \geq 1$ and elements $g_1,\ldots,g_\kappa \in G$ such that
$$C \subset \bigcup_{k=1}^\kappa \cU^{-1} g_k \; .$$
For every fixed $y \in Y$, we have
$$\cS \cap \pi_\ell^{-1}(\{y\}) \subset \bigcup_{k = 1}^\kappa \{(y,z) \mid z \in Y \cap (\cU^{-1} g_k \cdot y)\} \; .$$
By the injectivity of $\theta$, the sets in the union on the right hand side have at most one element. So,
$$\# (\cS \cap \pi_\ell^{-1}(\{y\})) \leq \kappa$$
for all $y \in Y$ and the lemma is proven.
\end{proof}

\subsection{Notation and conventions}\label{subsec.not}

Fix a lcsc unimodular group $G$ and fix an essentially free ergodic pmp action $G \actson (X,\mu)$. By Theorem \ref{thm.cross-section} and after discarding a $G$-invariant Borel set of measure zero, we get that $G \actson X$ is a free action that admits a cocompact cross section $Y \subset X$. We fix a Haar measure $\lambda$ on $G$ and we define the cross section equivalence relation $\cR$ and the probability measure $\nu$ on $Y$ as in Proposition \ref{prop.properties-cross-section-eq-rel}.

We fix a neighborhood $\cU$ of $e$ in $G$ such that $\theta : \cU \times Y \recht X : (g,y) \mapsto g \cdot y$ is injective. We also fix a compact set $K$ such that $K \cdot Y = X$. Since $K \times Y \recht X : (g,y) \recht g \cdot y$ is surjective and countable-to-one, we can choose a Borel right inverse $x \mapsto (\rho(x),\pi(x))$. We make this choice such that $\rho(g \cdot y) = g$ and $\pi(g \cdot y) = y$ for all $g \in \cU$ and $y \in Y$. Note that by construction, $\rho(x) \cdot \pi(x) = x$, so that $\pi(x) \in G \cdot x$ for all $x \in X$.

We put $M:= L^\infty(X) \rtimes G$ and  denote by $(u_g)_{g \in G}$ the canonical group of unitaries in the crossed product $L^\infty(X) \rtimes G$.

The von Neumann algebra $M$ is equipped with a normal semifinite faithful trace $\Tr$ satisfying, for all continuous compactly supported functions $f : G \recht \C$ and $a\in L^\infty(X)$,
\begin{equation}\label{eq.form-trace}
\Tr(a \lambda(f)) = f(e) \int_X a(x) \, d\mu(x) \quad\text{where}\quad \lambda(f) = \int_G f(g) u_g \; d\lambda(g) \; .
\end{equation}

Define
$$\cR_G := \{(x,y) \in X \times X \mid y \in G \cdot x\}$$
and  note that the map $X \times G \recht \cR_G : (x,g) \mapsto (x,g^{-1} \cdot x)$ is a Borel bijection. We equip $\cR_G$ with the push forward of the measure $\mu \times \lambda$ on $X \times G$. We can then identify the $M$-$M$-bimodule $L^2(M,\Tr)$ with $L^2(\cR_G)$ with the left and right module action being given by
$$\bigl((a u_g) \cdot \xi \cdot (u_h b)\bigr)(x,y) = a(x) \, \xi(g^{-1} \cdot x, h \cdot y)b(y) \quad\text{for all}\;\; a,b \in L^\infty(X), g,h \in G, (x,y) \in \cR_G \; .$$

\subsection{\boldmath The crossed product $L^\infty(X) \rtimes G$ is an infinite amplification of $L \cR$}\label{subsec.iso}

We keep the notations introduced in Section \ref{subsec.not}. As in Proposition \ref{prop.properties-cross-section-eq-rel}, define
$$Z := \{(x,y) \in X \times Y \mid x \in G \cdot y\}$$
and equip $Z$ with the $\sigma$-finite measure $\eta$ given by integrating w.r.t.\ $\mu$ the counting measure over the map $(x,y) \mapsto x$. Then $L^2(Z)$ is a right $L \cR$-module with right action given, for $\xi \in L^2(Z)$ and  $\vphi\in [[\cR]]$, by

$$(\xi \cdot u_\vphi)(x,y) =
\begin{cases} \xi(x,\vphi(y)) &\mbox{if } y \in D(\vphi) \\
0 & \mbox{otherwise } \end{cases}$$

We also define the left $G$-action on $L^2(Z)$ given by
$$(g \cdot \xi)(x,y) = \xi(g^{-1} \cdot x,y) \quad\text{for all}\;\; (x,y) \in Z, g \in G, \xi \in L^2(Z) \; .$$
Note that $L^2(Z)$ becomes a Hilbert $G$-$L \cR$-bimodule.

\begin{lemma}\label{lem.iso}
There exist
\begin{itemize}
\item a projection $p \in M$ with $\Tr(p) = \covol(Y)^{-1}$,
\item a unitary $U : L^2(M)p \recht L^2(Z)$,
\item a $*$-isomorphism $\phi : pMp \recht L \cR$,
\end{itemize}
such that $U(u_g \cdot \xi \cdot a) = g \cdot U(\xi) \cdot \phi(a)$ for all $g \in G$, $\xi \in L^2(M) p$ and $a \in pMp$.
\end{lemma}
\begin{proof}
Put $\cW := \cU \cdot Y$ and denote by $p_\cW \in L^\infty(X)$ the corresponding projection. Then $L^2(M)p_\cW = L^2(\cR_G \cap (X \times \cW))$. The map
$$\cR_G \cap (X \times \cW) \recht Z \times \cU : (x,x') \mapsto (x,\pi(x'),\rho(x'))$$
is Borel and bijective with inverse $(x,y,g) \mapsto (x,g \cdot y)$. Using Proposition \ref{prop.properties-cross-section-eq-rel}, we get that this map is measure preserving. So we define the unitary operator
$$V : L^2(M) p_\cW \recht L^2(Z \times \cU) : (V\xi)(x,y,g) = \xi(x,g\cdot y) \; .$$
Denote by $\rho : p_\cW M p_\cW \recht B(L^2(M) p_\cW)$ the $*$-antihomomorphism given by the right action of $p_\cW M p_\cW$. Similarly, denote by $\rho : L \cR \recht B(L^2(Z))$ the $*$-antihomomorphism given by the right action of $L \cR$. A direct computation shows that
$$V \rho(p_\cW a u_g p_\cW) V^* \in \rho(L \cR) \ovt B(L^2(\cU)) \quad\text{for all}\;\; a \in L^\infty(X), g \in G \; .$$
So it follows that $V \rho(p_\cW M p_\cW) V^* \subset \rho(L \cR) \ovt B(L^2(\cU))$. Denote by $\gamma : M \recht B(L^2(M)p_\cW)$ the $*$-homomorphism given by the left action of $M$. A direct computation also shows that $V \gamma(a u_g) V^*$ commutes with $\rho(L \cR) \ovt B(L^2(\cU))$ for all $a \in L^\infty(X)$, $g \in G$. Since $$B(L^2(M)p_\cW) \cap \gamma(M)' = \rho(p_\cW M p_\cW) \; ,$$
we find that $V^* (\rho(L \cR) \ovt B(L^2(\cU))) V \subset \rho(p_\cW M p_\cW)$. So we have proven that
$$V \rho(p_\cW M p_\cW) V^* = \rho(L \cR) \ovt B(L^2(\cU)) \; .$$
The left hand side is equipped with the restriction of the trace $\Tr$ on $M$. The right hand side is equipped with the tensor product of the natural tracial state on $L \cR$ and the semifinite trace $\Tr$ on $B(L^2(\cU))$ that is normalized such that the trace of a minimal projection equals $1$. Since $M$ is a II$_\infty$ factor both traces are a multiple of each other under the isomorphism $\Ad V$. We must determine this multiple. To do so, choose a nonempty open subsets $\cU_0 \subset \cU$ and choose a continuous compactly supported function $f : G \recht \C$ such that $f(e) = 1$ and $g \cU_0 \subset \cU$ for all $g \in \supp f$. Put $\cW_0 := \cU_0 \cdot Y$. Define the element $S \in p_\cW M p_\cW$ given by
$$S := \int_G f(g) \, u_g p_{\cW_0} \; d\lambda(g) \; .$$
One computes that $V \rho(S) V^* = 1 \ot T$, where $T \in B(L^2(\cU))$ is given by
$$T := \int_G f(g) \, p_{\cU_0} \lambda_g^* \; d\lambda(g)$$
and where $(\lambda_g)_{g \in G}$ denotes the left regular representation of $G$ on $L^2(G)$. We have $\Tr(S) = \mu(\cW_0)$ and $(\tau \ot \Tr)(1 \ot T) = \lambda(\cU_0)$. So, $\Ad V$ induces a $*$-isomorphism between $\rho(p_\cW M p_\cW)$ and $\rho(L \cR) \ovt B(L^2(\cU))$ that scales the trace with the factor $\covol Y$.

Choose a minimal projection $q \in B(L^2(\cU))$ and denote by $p \in p_\cW M p_\cW$ the projection such that $V \rho(p) V^* = 1 \ot q$. We get that $\Tr(p) = \covol(Y)^{-1}$. We find the $*$-isomorphism $\phi : p M p \recht L \cR$ such that $V \rho(a) V^* = \rho(\phi(a)) \ot q$ for all $a \in p M p$. The restriction of $V$ to $L^2(M) p$ yields the required unitary $U : L^2(M)p \recht L^2(Z)$.
\end{proof}

\subsection{\boldmath Proof of Theorem \ref{thm.mainA}}

It suffices to prove Theorem \ref{thm.mainA} for one particular choice of cross section. Indeed, if $Y$ and $Y'$ are two cross sections, with corresponding cross section equivalence relations $\cR$ and $\cR'$, then by Proposition \ref{prop.properties-cross-section-eq-rel}, $\cR$ and $\cR'$ are stably orbit equivalent with compression constant $\covol(Y)/\covol(Y')$. So, by \cite[Corollaire 5.6]{Ga01}, we find that
$$\covol(Y)^{-1} \, \beta_n^{(2)}(\cR) = \covol(Y')^{-1} \, \beta_n^{(2)}(\cR') \; .$$
So Theorem \ref{thm.mainA} holds for the cross section $Y$ if and only if it holds for the cross section $Y'$.

Therefore we can fix a cocompact cross section $Y$ and use the notations introduced in Section \ref{subsec.not}. We also use the Hilbert $G$-$L \cR$-bimodule $L^2(Z)$ introduced in the beginning of Section \ref{subsec.iso} and we use cohomology of $G$ with coefficients in $L^2(Z)$, in the sense of Definition \ref{def.bar-res}.

\begin{step}\label{step1}
We have
\begin{align*}
& \beta^n_{(2)}(G) \leq \covol(Y)^{-1} \, \dim_{L \cR} H^n(G,L^2(Z)) \quad\text{and}\\
& \betar^n_{(2)}(G) = \covol(Y)^{-1} \, \dim_{L \cR} \Hred^n(G,L^2(Z)) \; .
\end{align*}
\end{step}
\begin{proof}
Put $N = LG$ and $M = L^\infty(X) \rtimes G$. Note that we have a natural trace preserving inclusion $N \subset M$. Using the Connes tensor product, as explained in the beginning of Section \ref{subsec.coeff-change}, we have
$$L^2(G) \ovt_N L^2(M) = L^2(M) \; .$$
So, it follows from Proposition \ref{prop.change-coeff} that
$$\beta^n_{(2)}(G) \leq \dim_M H^n(G,L^2(M)) \quad\text{and}\quad \betar^n_{(2)}(G) = \dim_M \Hred^n(G,L^2(M)) \; .$$
Take a projection $p \in M$ satisfying the conclusions of Lemma \ref{lem.iso}. Since $M$ is a factor, the central support of $p$ in $M$ equals $1$. So from Lemma \ref{lem.realize-semifinite-dim}, we get that
\begin{align*}
& \dim_M H^n(G,L^2(M)) = \covol(Y)^{-1} \, \dim_{pMp} H^n(G,L^2(M)p) \quad\text{and}\\
& \dim_M \Hred^n(G,L^2(M)) = \covol(Y)^{-1} \, \dim_{pMp} \Hred^n(G,L^2(M)p) \; .
\end{align*}
Using the unitary $U : L^2(M)p \recht L^2(Z)$ and the isomorphism $\vphi : pMp \recht L \cR$, we get that
\begin{align*}
& \dim_{pMp} H^n(G,L^2(M)p) = \dim_{L \cR} H^n(G,L^2(Z)) \quad\text{and}\\
& \dim_{pMp} \Hred^n(G,L^2(M)p) = \dim_{L \cR} \Hred^n(G,L^2(Z)) \; .
\end{align*}
So, step \ref{step1} is proven.
\end{proof}

\begin{step}\label{step2}
We have
$$\dim_{L \cR} H^n(G,L^2(Z)) = \dim_{L \cR} \Hred^n(G,L^2(Z)) = \beta_n^{(2)}(\cR) \; .$$
\end{step}
\begin{proof}
We fix an increasing sequence of compact subsets $K_k \subset G$ whose interiors cover $G$. Since $K_k$ is compact, Lemma \ref{lem.compact-finite} provides finite subsets $\cF_k \subset [[\cR]]$ such that
$$(K_k \cdot y) \cap Y = \cF_k \cdot y \quad\text{for all}\;\; y \in Y \; .$$
Here we use the notation $\cF_k \cdot y$ to denote the set of points of the form $\vphi(y)$ with $\vphi \in \cF_k$ and $y \in D(\vphi)$.
We denote by $\Lambda$ the set of all compositions of elements in $\cF_k \cup \cF_k^{-1} \cup \{\id\}$, $k \in \N$. We write $\Lambda$ as an increasing union of finite subsets $\Lambda_k \subset \Lambda$ with $\id \in \Lambda_0$.

For every $n \geq 0$, define the set
$$\Xi^{(n)} := \{(x,y_0,\ldots,y_n,z) \in X \times Y^{n+2} \mid (\pi(x),y_0,\ldots,y_n,z) \in \cR^{(n+2)}, \forall i \neq j : y_i  \neq y_j  \}$$
and the sequence of subsets $\Xi_k^{(n)} \subset \Xi^{(n)}$ given by
\begin{align*}
\Xi_k^{(n)} := \{(x,y_0,\ldots,y_n,z) \in \Xi^{(n)} \mid \; & \text{there exist}\;\; \vphi_0,\ldots,\vphi_n \in \Lambda_k \;\; \text{such that}\;\; \pi(x) \in D(\vphi_i) \\ & \text{and}\;\; y_i = \vphi_i(\pi(x)) \;\;\text{for all}\;\; i=0,\ldots,n \} \; .
\end{align*}
We equip $\Xi^{(n)}$ with the $\sigma$-finite measure $\eta^{(n)}$ given by integrating w.r.t.\ $\mu$ the counting measure over the projection onto the first coordinate $(x,y_0,\ldots,y_n,z) \mapsto x$.
For every $n \geq 0$, we consider the Fr\'{e}chet space $D^n := L^2\locXi(\Xi^{(n)})$ of functions that are square integrable on each of the $\Xi_k^{(n)}$, $k \in \N$.
We turn $D^n$ into a Fr\'{e}chet $G$-$L \cR$-bimodule using
$$(g \cdot \xi \cdot u_\vphi)(x,y_0,\ldots,y_n,z) = \xi(g^{-1} \cdot x,y_0,\ldots,y_n,\vphi(z))$$
for all $g \in G$, $\xi \in D^n$, $\vphi \in [[\cR]]$ and $(x,y_0,\ldots,y_n,z) \in \Xi^{(n)}$. We define the complex of Fr\'{e}chet $G$-$L \cR$-modules given by
\begin{equation}\label{eq.ourcomplex}
0 \longrightarrow L^2(Z) \pijl{d_{-1}} D^0 \pijl{d_0} D^1 \pijl{d_1} D^2 \pijl{d_2} \cdots
\end{equation}
with the coboundary operators given by
\begin{align*}
& (d_{-1} \xi)(x,y_0,z) = \xi(x,z) \; , \\
& (d_n \xi)(x,y_0,\ldots,y_{n+1},z) = \sum_{i=0}^{n+1} (-1)^i \xi(x,y_0,\ldots,\widehat{y_i},\ldots,y_n,z) \; .
\end{align*}
We claim that the complex \eqref{eq.ourcomplex} of Fr\'{e}chet $L\cR$-modules is strongly exact in the sense of Definition \ref{def.strong-exact-acyclic}. For this, it suffices to define
\begin{align*}
& S_0 : D^0 \recht L^2(Z) : (S_0 \xi)(x,z) = \xi(x,\pi(x),z) \; , \\
& S_n : D^n \recht D^{n-1} : (S_n \xi)(x,y_1,\ldots,y_n,z) = \xi(x,\pi(x),y_1,\ldots,y_n,z) \; .
\end{align*}
A direct computation now gives $S_{n+1}\circ d_n+d_{n-1}\circ S_n=\id_{D_n}$.
We next claim that for every $n \geq 0$, the Fr\'{e}chet $G$-$L \cR$-bimodule $D^n$ is strongly acyclic in the sense of Definition \ref{def.strong-exact-acyclic}. Using Lemma \ref{lem.make-acyclic}, it suffices to prove that $D^n$ is of the form $L^2\loc(G,C^n)$ for a certain Fr\'echet $L \cR$-module $C^n$. To prove this last statement, define as before
$$\cR_G := \{(x,x') \in X \times X \mid x \in G \cdot x'\} \; .$$
Since the action of $G$ on $X$ is free, we can uniquely define the Borel map $\Omega : \cR_G \recht G$ such that $\Om(x,x') \cdot x' = x$ for all $(x,x') \in \cR_G$.

Define the sets $\Sigma^{(n)} \subset \cR^{(n+1)}$ given by \eqref{eq.sets-Sigma}. We equip $\Sigma^{(n)}$ with the measure given by restricting $\nu^{(n+1)}$. The maps
$$\theta_n : G \times \Sigma^{(n)} \recht \Xi^{(n)} : (g,y_0,\ldots,y_n,z) \mapsto (g \cdot y_0,y_0,\ldots,y_n,z)$$
are Borel and bijective with the inverse given by
$$\theta_n^{-1}(x,y_0,\ldots,y_n,z) = (\Om(x,y_0),y_0,\ldots,y_n,z) \; .$$
Because of Proposition \ref{prop.properties-cross-section-eq-rel}, we have $(\theta_n)_*(\lambda \times \nu^{(n+1)}) = \covol(Y) \eta^{(n)}$.

Since we have chosen a cocompact cross section, the increasing sequence of subsets of $G \times \Sigma^{(n)}$ given by $\theta_n^{-1}(\Xi^{(n)}_k)$ is cofinal with the increasing sequence of subsets $K_k \times \Sigma_k^{(n)}$, where $\Sigma_k^{(n)}$ was defined in \eqref{eq.subsets-Sigma}. So we indeed find a $G$-$L\cR$-linear bijective homeomorphism $D^n \cong L^2\loc(G,C^n)$, with $C^n = L^2\locS(\Sigma^{(n)})$. It then follows from Lemma \ref{lem.make-acyclic} that the $G$-$L \cR$-bimodule $D^n$ is strongly acyclic.

Since moreover the complex in \eqref{eq.ourcomplex} is strongly exact, it follows from Proposition \ref{prop.resolution} that there are $L \cR$-linear isomorphisms
\begin{equation}\label{eq.good-iso}
H^n(G,L^2(Z)) \cong H^n(\cC) \quad\text{and}\quad \Hred^n(G,L^2(Z)) \cong \Hred^n(\cC) \; ,
\end{equation}
where $\cC$ is the complex of Fr\'{e}chet $L \cR$-modules given by
\begin{equation}\label{eq.complex-fixed-point}
(D^0)^G \pijl{d_0} (D^1)^G \pijl{d_1} (D^2)^G \pijl{d_3} \cdots \; .
\end{equation}
Using the isomorphism $D^n \cong L^2\loc(G,C^n)$ that we obtained in the previous paragraph, the complex in \eqref{eq.complex-fixed-point} is isomorphic with the complex
$$L^2\locS(\Sigma^{(0)}) \pijl{d_0} L^2\locS(\Sigma^{(1)}) \pijl{d_1} L^2\locS(\Sigma^{(2)}) \pijl{d_3} \cdots$$
that we considered in Proposition \ref{prop.compute-Betti-R}. So Proposition \ref{prop.compute-Betti-R} gives us that
$$\dim_{L \cR} H^n(\cC) = \dim_{L \cR} \Hred^n(\cC) = \beta_n^{(2)}(\cR) \; .$$
Using \eqref{eq.good-iso}, step \ref{step2} is proven.
\end{proof}

\begin{proof}[{\bf End of the proof of Theorem \ref{thm.mainA}}] Combining steps \ref{step1} and \ref{step2} with the obvious inequality $\betar^n_{(2)}(G) \leq \beta^n_{(2)}(G)$, we get that
$$\covol(Y)^{-1} \, \beta_n^{(2)}(\cR) = \betar^n_{(2)}(G) \leq \beta^n_{(2)}(G) = \covol(Y)^{-1} \, \beta_n^{(2)}(\cR) \; .$$
So the middle inequality must also be an equality and Theorem \ref{thm.mainA} is proven.
\end{proof}

\subsection{\boldmath Proof of Theorem \ref{thm.mainB}}

Fix Haar measures $\lambda_H$ on $H$ and $\lambda_G$ on $G$. Fix a Borel cross section $\theta : G/H \recht G$ satisfying $\theta(eH) = e$ and denote by $\lambda_{G/H}$ the $G$-invariant measure on $G/H$ given by \eqref{eq.invariant-quotient}. We define the $1$-cocycle
$$\om : G \times G/H \recht H : g \, \theta(k H) = \theta(gk H) \, \om(g,kH) \quad\text{for all}\;\; g,k \in G \; .$$
Note that $\om(h,eH) = h$ for all $h \in H$.

By Remark \ref{exist-action}, we can choose an essentially free ergodic pmp action $H \actson (X,\mu)$. Define $X' = G/H \times X$ and equip $X'$ with the probability measure $\mu' := \covol(H)^{-1} \, (\lambda_{G/H} \times \mu)$. Define the induced action $G \actson (X',\mu')$ given by
$$g \cdot (kH, x) = (gkH, \om(g,kH) \cdot x) \; .$$
Note that $G \actson (X',\mu')$ is an essentially free ergodic pmp action.

Fix a cross section $Y \subset X$ for the action $H \actson (X,\mu)$. Define $Y' := \{eH\} \times Y$. We claim that $Y' \subset X'$ is a cross section for the action $G \actson (X',\mu')$. Since $H \cdot Y' = \{eH\} \times (H \cdot Y)$, we get that $G \cdot Y' = G/H \times (H \cdot Y)$, which is conegligible in $X'$. Take a neighborhood $\cU$ of $e$ in $H$ such that the map $\cU \times Y \recht X : (h,y) \mapsto h \cdot y$ is injective. Choose a neighborhood $\cU'$ of $e$ in $G$ such that $(\cU')^{-1} \cU' \cap H \subset \cU$. To prove our claim, it suffices to prove that the map $\cU' \times Y' \recht X' : (g,y') \mapsto g \cdot y'$ is injective.

So assume that $g,k \in \cU'$ and $y,z \in Y$ such that $g \cdot (eH,y) = k \cdot (eH,z)$. Then $gH = kH$ and by our choice of $\cU'$, we get that $k = g h$ for some $h \in \cU$. But then $y = h \cdot z$, so that $y=z$ and $h=e$. Then also $g = k$ and the required injectivity is proven.

The cross section equivalence relations on $Y'$ and $Y$ are identical. Only their canonical covolumes differ. As in Proposition \ref{prop.properties-cross-section-eq-rel}, define $Z \subset X \times Y$ with its natural measure $\eta$ and denote by $\nu$ the natural probability measure on $Y$. Define $\nu'$ on $Y'$ such that $\nu = \nu'$ under the obvious identification of $Y$ and $Y'$. Finally define $Z' \subset X' \times Y'$, again with its natural measure $\eta'$. The cross section $\theta$ induces a bijective Borel map $G \recht G/H \times H$. We then have the obvious maps
$$(G \times Y', \lambda_G \times \nu') \recht (G/H \times H \times Y, \lambda_{G/H} \times \lambda_H \times \nu) \recht (G/H \times Z, \lambda_{G/H} \times \eta) \recht (Z',\eta') \; .$$
The first one is measure preserving, the second one scales the measure with a factor $\covol Y$ and the last one scales the measure with a factor $\covol(H)$. We conclude that
$$\covol(Y') = \covol(H) \cdot \covol(Y) \; .$$
Since the cross section equivalence relations on $Y'$ and $Y$ are identical, Theorem \ref{thm.mainB} follows from this formula.\hfill\qedsymbol

\subsection{\boldmath Vanishing results~: proof of Theorem \ref{thm.mainC}}

1. If $G$ is compact and $\lambda$ is a Haar measure on $G$, we consider the action of $G$ on itself, equipped with the probability measure $\lambda(G)^{-1} \cdot \lambda$. Then $\{e\}$ is a cross section. It has covolume $\lambda(G)$ and the cross section equivalence relation is, obviously, the trivial equivalence relation on one point. So 1 follows.

2. Take $G$ a lcsc unimodular amenable group that is noncompact. Take any essentially free ergodic pmp action $G \actson (X,\mu)$ with cross section equivalence relation $\cR$. By \ref{prop.properties-cross-section-eq-rel}, $\cR$ is ergodic, amenable and has infinite orbits a.e. So by \cite{CFW81}, $\cR$ is the orbit equivalence relation of an essentially free ergodic pmp action of $\Z$. Then $\beta_n^{(2)}(\cR) = 0$ for all $n \geq 0$ and 2 follows.

3 and 4. First make the following general observation. If $G$ is a lcsc group with left Haar measure $\lambda_G$, then the modular function $\Delta_G : G \recht \R_*^+$ is defined such that $\lambda_G \circ \Ad g = \Delta_G(g)^{-1} \cdot \lambda_G$. Whenever $H \lhd G$ is a closed normal subgroup, the Haar measure on the quotient group is $G$-invariant and therefore $\Delta_H = (\Delta_G)_{|H}$; cf.\ \cite[Corollary B.1.7]{BHV08}.

In particular, if $G$ is unimodular, also $H$ is unimodular. Moreover, the uniqueness of the Haar measure on $H$ allows to define the homomorphism $\al : G \recht \R_*^+$ such that $\lambda_H \circ \Ad g = \al(g)^{-1} \cdot \lambda_H$, and using \eqref{eq.invariant-quotient}, one deduces that
\begin{equation}\label{eq.formulas-modular-function}
\Delta_G(g) = \Delta_{G/H}(gH) \, \al(g) \;\;\text{for all}\;\; g \in G \; .
\end{equation}
In the case where $G$ is nonunimodular and $H = \Ker \Delta_G$, it follows that $H$ is unimodular and that $H$ is noncompact. Indeed if $H$ would be compact, we have $\al(g) = 1$ for all $g \in G$ and also $\Delta_{G/H} = 1$ because $G/H$ is abelian. So \eqref{eq.formulas-modular-function} would then imply that $G$ is unimodular.

We next prove 3 and 4 in the special case where also $G/H$ is unimodular. So fix a unimodular lcsc group $G$ and a closed normal subgroup $H \lhd G$. Assume that $G/H$ is unimodular and assume that $H$ is noncompact. Note that both in 3 and 4, we assume that $\beta^0_{(2)}(H) = 0$, so that $H$ is indeed noncompact because of 1.
By Remark \ref{exist-action}, we can choose a free mixing pmp action $G \actson (X,\mu)$. Similarly choose a free ergodic pmp action $G/H \actson (X',\mu')$. Denote by $\pi \colon G \recht G/H$ the quotient homomorphism and define the action $G \actson X \times X'$ given by $g \cdot (x,x') = (g \cdot x, \pi(g) \cdot x')$. We write $(X\dpr,\mu\dpr) := (X \times X',\mu \times \mu')$. Since $G \actson (X,\mu)$ is mixing and $G/H \actson (X',\mu')$ is ergodic, the action $G \actson (X\dpr,\mu\dpr)$ is ergodic as well. Since the action $G \actson (X,\mu)$ is mixing and $H$ is noncompact, the restricted action $H \actson (X,\mu)$ is still ergodic.

Choose a cross section $Y \subset X$ for the action $H \actson (X,\mu)$ and denote by $\cR$ the associated cross section equivalence relation. Choose a cross section $Y' \subset X'$ for the action $G/H \actson (X',\mu')$ and denote by $\cR'$ the associated cross section equivalence relation. We claim that $Y\dpr := Y \times Y'$ is a cross section for the action $G \actson (X\dpr,\mu\dpr)$. To prove this claim, choose a neighborhood $\cU$ of $e$ in $H$ such that the action map $\cU \times Y \recht X$ is injective. Also choose a neighborhood $\cU'$ of $eH$ in $G/H$ such that the action map $\cU' \times Y' \recht X'$ is injective. Take a neighborhood $\cU\dpr$ of $e$ in $G$ such that $\pi(\cU) \subset \cU'$ and such that $H \cap (\cU\dpr)^{-1} \, \cU\dpr  \subset \cU$. It follows that the action map $\cU\dpr \times Y\dpr \recht X\dpr$ is injective. Indeed, if $g,k \in \cU\dpr$ and $g \cdot (x,x') = k \cdot (y,y')$ for some $(x,x')$, $(y,y')$ in $Y\dpr$, we first conclude that $\pi(g) \cdot x' = \pi(k) \cdot y'$. Since $\pi(g),\pi(k) \in \cU'$, it follows that $\pi(g) = \pi(k)$ and $x' = y'$. So $k = g h$ with $h \in H \cap (\cU\dpr)^{-1} \, \cU\dpr$. So $h \in \cU$. But also $x = h \cdot y$, so that $x=y$ and $h=e$. This proves the injectivity of the action map $\cU\dpr \times Y\dpr \recht X\dpr$.

To conclude the proof of the claim, we have to show that $G \cdot Y\dpr$ is conegligible in $X\dpr$. Define $X_0 := H \cdot Y$. Then $X_0$ is a conegligible Borel subset of $X$ and $X_0 \times Y' = H \cdot Y\dpr \subset G \cdot Y\dpr$. By the Fubini theorem, a.e.\ $x \in X$ has the property that $g^{-1} \cdot x \in X_0$ for a.e.\ $g \in G$. Since $X_0 \times Y' \subset G\cdot Y\dpr$, we conclude that a.e.\ $x \in X$ has the property that
$$(x,\pi(g) \cdot y') = g \cdot (g^{-1}\cdot x, y') \in G\cdot Y\dpr \quad\text{for all $y' \in Y'$ and a.e.\ $g \in G$.}$$
Using \eqref{eq.first-push}, it follows that a.e.\ $x \in X$ has the property that $(x,x') \in G \cdot Y\dpr$ for a.e.\ $x' \in X'$. Again using the Fubini theorem, it follows that $G \cdot Y\dpr$ is conegligible. So we have proven the claim that $Y\dpr$ is a cross section for the action $G \actson (X\dpr,\mu\dpr)$.

Denote by $\cR\dpr$ the cross section equivalence relation on $Y\dpr$. Denote by $\cR \times \id$ the equivalence relation on $Y\dpr$ defined as $\{((y,y'),(z,z')) \in Y\dpr \times Y\dpr \mid (y,z) \in \cR , y'=z'\}$. Define the map
$$\gamma : \cR\dpr \recht \cR' : \gamma((y,y'),(z,z')) = (y',z') \; .$$
Then $\gamma$ is a surjective homomorphism of equivalence relations and the kernel of $\gamma$ is given by $\cR \times \id$. Moreover, since $\cR$ is an ergodic equivalence relation, $\gamma$ is strongly surjective in the sense of \cite[Definition 3.7]{ST07}. In \cite[Theorems 1.3 and 1.5]{ST07}, it is shown that the equivalence relation version (even discrete measured groupoid version) of properties 3 and 4 holds for strongly normal subequivalence relations. So by the above construction, properties 3 and 4 hold whenever $G/H$ is unimodular.

We finally deduce the general case. We still denote by $\pi : G \recht G/H$ the quotient homomorphism. Denote $G_0 := \Ker(\Delta_{G/H} \circ \pi)$. Then $H \lhd G_0$ is a closed normal subgroup and $G_0/H = \Ker \Delta_{G/H}$ is unimodular. Also $G_0 \lhd G$ is a closed normal subgroup and the quotient $G/G_0$ is abelian, hence unimodular. If $\beta^n_{(2)}(H) = 0$ for all $0 \leq n \leq d$, we apply twice the already proven special case of property 3 and conclude that $\beta^n_{(2)}(G) = 0$ for all $0 \leq n \leq d$.

Finally assume that $\beta^n_{(2)}(H) = 0$ for all $0 \leq n \leq d$ and that $\beta^{d+1}_{(2)}(H) < \infty$. Also assume that $G/H$ is noncompact and nonunimodular. As we explained just after \eqref{eq.formulas-modular-function}, it follows that $G_0/H$ is noncompact and unimodular. So by the already proven special case of 4, we get that $\beta^n_{(2)}(G_0) = 0$ for all $0 \leq n \leq d+1$. Then applying property 3 to the normal subgroup $G_0$ of $G$, we conclude that $\beta^n_{(2)}(G) = 0$ for all $0 \leq n \leq d+1$.\hfill\qedsymbol

\subsection{\boldmath Proof of Corollary \ref{cor.corD}}

\begin{proposition}\label{prop.finite-first-Betti}
Let $G$ be a lcsc unimodular group and $G \actson (X,\mu)$ an essentially free ergodic pmp action. Let $Y \subset X$ be a cross section and denote by $\cR$ the cross section equivalence relation.
If $G$ is compactly generated, then $\cR$ has finite cost in the sense of \cite[D\'{e}finition I.5]{Ga99}. In particular, $\beta^1_{(2)}(G) < \infty$.
\end{proposition}
\begin{proof}
Fix an essentially free ergodic pmp action $G \actson (X,\mu)$. By \cite[Invariance II.2]{Ga99}, the cost of an ergodic countable pmp equivalence relation is preserved under stable orbit equivalence. So using Proposition \ref{prop.properties-cross-section-eq-rel}, it suffices to prove the proposition for a cocompact cross section $Y \subset X$. Discarding a $G$-invariant Borel set of measure zero, we may assume that there exists a compact subset $K \subset G$ such that $K \cdot Y = X$.

Take a compact subset $C \subset G$ that generates $G$ as a group. Take $C$ such that $C = C^{-1}$ and put $L = K^{-1} C K$. Since $L$ is compact, Lemma \ref{lem.compact-finite} provides us with a finite subset $\cF \subset [[\cR]]$ satisfying $Y \cap (L \cdot y) \subset \cF \cdot y$ for all $y \in Y$. We prove that $\cF$ is a graphing for $\cR$, meaning that for all $(y,z) \in \cR$, there exist $\vphi_1,\ldots,\vphi_m \in \cF$ such that $y = (\vphi_m \circ \cdots \circ \vphi_1)(z)$. To prove this statement, fix $(y,z) \in \cR$. Since $G$ is generated by $C$ and $C = C^{-1}$, take $g_1,\ldots,g_m \in C$ such that $y = (g_m \cdots g_1) \cdot z$. Since $X = K \cdot Y$, we can take $h_1,\ldots,h_{m-1} \in K$ such that
$$z_i := h_i^{-1} \cdot ((g_i \cdots g_1) \cdot z) \quad\text{belongs to}\;\; Y \; .$$
Put $h_0 = h_m = e$ and put $z_0 = z$, $z_m = y$. Finally put $k_i := h_i^{-1} g_i h_{i-1}$, for all $i=1,\ldots,m$. By construction $k_i \cdot z_{i-1} = z_i$ for all $i=1,\ldots,m$. Since $k_i \in L$ and $z_{i-1}, z_i \in Y$, we can take $\vphi_i \in \cF$ such that $z_i = \vphi_i(z_{i-1})$. We have proven that $y = (\vphi_m \circ \cdots \circ \vphi_1)(z)$. So $\cF$ is a graphing for $\cR$.

Since $\cF$ is a graphing for $\cR$ and since $\cF$ is a finite set, it follows that $\cR$ has finite cost.

We finally deduce that $\beta^1_{(2)}(G) < \infty$. If $G$ is compact, then $\beta^1_{(2)}(G) = 0$ by Theorem \ref{thm.mainC}. If $G$ is noncompact but compactly generated, we know from Proposition \ref{prop.properties-cross-section-eq-rel} that $\cR$ has infinite orbits a.e.\ and we proved above that $\cR$ has finite cost. Using \cite[Corollaire 3.23]{Ga01}, we get that
$$\beta_1^{(2)}(\cR) \leq \beta_0^{(2)}(\cR) + \operatorname{cost}(\cR) - 1 = \operatorname{cost}(\cR) - 1 < \infty \; .$$
Theorem \ref{thm.mainA} then implies that $\beta^1_{(2)}(G) < \infty$.
\end{proof}

It is now immediate to deduce Corollary \ref{cor.corD} from Theorem \ref{thm.mainC}.

\begin{proof}[Proof of Corollary \ref{cor.corD}]
Since $H$ is noncompact and compactly generated, it follows from Propositions \ref{prop.properties-cross-section-eq-rel} and \ref{prop.finite-first-Betti} that $\beta^0_{(2)}(H) = 0$ and $\beta^1_{(2)}(H) < \infty$. Since $G/H$ is noncompact, it follows from Theorem \ref{thm.mainC} that $\beta^1_{(2)}(G) = 0$.
\end{proof}

\titleformat{\section}{\Large\bf\boldmath}{Appendix \thesection.}{2ex}{}{}
\renewcommand{\thesection}{A}
\section{Some dimension theory for $M$-modules}\label{appA}

Let $(M,\tau)$ be a von Neumann algebra with separable predual equipped with a faithful normal tracial state. Any separable Hilbert $M$-module $H$ is isomorphic with $p (\ell^2(\N) \ovt L^2(M))$ for some projection $p \in \B(\ell^2(\N)) \ovt M$ that is uniquely determined up to equivalence of projections. The Murray-von Neumann dimension of the Hilbert $M$-module $H$ is defined as $(\Tr \ot \tau)(p)$. Following L\"{u}ck, see \cite{Lu97} or \cite[Section 6.1]{Lu02}, an arbitrary (algebraic) $M$-module $H$ has a dimension $\dim_M H$ defined by the formula
\begin{equation}\label{eq.def-dim}
\begin{split}
\dim_M H := \sup \{(\Tr \ot \tau)(p) \mid \; & p \in \M_n(\C) \ot M \;\;\text{is a projection and there exists} \\ & \text{an injective $M$-linear map}\;\; p(\C^n \ot M) \recht H\} \; .
\end{split}
\end{equation}
The three basic properties of L\"{u}ck's dimension function are given in the following theorem.

\begin{theorem}\label{thm.basic}
The dimension function $\dim_M$ satisfies the following properties.
\begin{enumerate}
\item (\cite[Theorem 6.24]{Lu02}) For every projection $p \in \B(\ell^2(\N)) \ovt M$, L\"{u}ck's dimension coincides with the Murray-von Neumann dimension, i.e.\
$$\dim_M (p (\ell^2(\N) \ovt L^2(M))) = (\Tr \ot \tau)(p) \; .$$

\item (\cite[Theorem 6.7]{Lu02}) If $0 \recht K \recht H \recht L \recht 0$ is an exact sequence of $M$-modules, then $\dim_M H = \dim_M K + \dim_M L$.  We refer to this as the rank theorem.

\item (\cite[Theorem 2.4]{Sa03}) An $M$-module $H$ has $\dim_M H = 0$ if and only if for every $\xi \in H$ and every $\eps > 0$, there exists a projection $p \in M$ with $\tau(p) > 1 - \eps$ and $\xi p = 0$.
\end{enumerate}
\end{theorem}

We need several other properties of the dimension function $\dim_M$.
For the convenience of the reader, we provide detailed arguments.

Throughout this appendix, we fix a von Neumann algebra $M$ with separable predual equipped with a faithful normal tracial state $\tau$.

\subsection{Generalities}

\begin{definition}
Let $H$ be an $M$-module and $H_0 \subset H$ an $M$-submodule. We say that $H_0$ is \emph{rank dense} in $H$ if for every $x \in H$ and every $\eps > 0$, there exists a projection $p \in M$ with $\tau(p) > 1-\eps$ and $xp \in H_0$.

We say that an $M$-module $H$ is of \emph{rank zero} if for every $x \in H$ and every $\eps > 0$, there exists a projection $p \in M$ with $\tau(p) > 1-\eps$ and $xp = 0$.

Let $K,H$ be $M$-modules and $T : K \recht H$ an $M$-linear map. We call $T$ an \emph{isomorphism in rank} if $\Ker T$ is an $M$-module of rank zero and if $\image T$ is rank dense in $H$.
\end{definition}

Using Theorem \ref{thm.basic}, one immediately gets the following result.

\begin{proposition} \label{prop.rank-iso-dim}
The dimension function satisfies the following properties.
\begin{enumerate}
\item If $H_0 \subset H$ is a rank dense $M$-submodule, then $\dim_M H_0 = \dim_M H$.
\item If $T : K \recht H$ is an isomorphism in rank, than $\dim_M K = \dim_M H$.
\end{enumerate}
\end{proposition}

\begin{definition}
A Hilbert $M$-module $H$ is said to be \emph{finitely generated} if there exist finitely many $\xi_1,\ldots,\xi_n \in H$ such that $\xi_1 M + \cdots + \xi_n M$ is dense in $H$. Equivalently, $H$ is isomorphic with $p (\C^n \ot L^2(M))$ for some projection $p \in \M_n(\C) \ot M$.
\end{definition}

\begin{lemma}\label{lem.rank-dense}
Let $H$ be a finitely generated Hilbert $M$-module and $K \subset H$ an $M$-submodule. If $K$ is dense in $H$, then $K$ is rank dense in $H$.
\end{lemma}
\begin{proof}
Replacing $M$ by $\M_n(\C) \ot M$, we may assume that $H = p L^2(M)$ for some projection $p \in M$. Define
$$\cP := \{a \in p M^+ p \mid a M \subset K \}\; .$$
Whenever $\xi \in K$ and $p_k$ equals the spectral projection $\chi_{(1/k,k)}(\xi \xi^*)$, we have $p_k \xi \in M$ and hence $\xi \xi^* p_k M \subset K$. Since $\xi \xi^* p_k M = p_k M$, we get that $p_k \in \cP$. If $k \recht \infty$, then $p_k$ increases to the left support projection of $\xi$.
Further, $\cP$ is closed under sums and under taking spectral projections $\chi_{(1/k,k)}(a)$.

Using this, we first show that $p$ can be approximated in the strong operator topology with projections from $\cP$.
Since $K$ is dense in $H$, we can find a sequence $\xi_n \in K$ such that $\|\xi_n - p\|_2\to 0.$
Denoting by $q_n$ the left support of $\xi_n$ we have $p=\vee_n q_n$ and, by what was just proven, each $q_n$ is a countable union of projections from $\cP$.  It therefore suffices to show that $r_1\vee\dots \vee r_k \in\overline{\text{Proj}(\cP)}^{\text{SOT}}$ whenever $r_1,\dots, r_k\in \text{Proj}(\cP)$. But
\[
r_1\vee\dots \vee r_k = \text{left support of }r_1+ \dots + r_k= \lim_{m\to\infty}  \underbrace{\chi_{(1/m,m)}(r_1 +\dots +r_k)}_{\in \cP},
\]
so $r_1\vee\dots \vee r_k \in\overline{\text{Proj}(\cP)}^{\text{SOT}}$ as desired. We may therefore choose a sequence of projections $p_k\in \cP$ converging strongly to $p$.

Choose $\eta \in H$ and $\eps > 0$. Take a projection $q_0 \in M$ such that $\tau(q_0) > 1-\eps/2$ and $\eta q_0 \in M$. Take $k$ large enough such that $\tau(p-p_k) < \eps /2$. Denote by $q_1$ the right support projection of $(p-p_k) \eta q_0$. Then $q_1 \leq q_0$ and $\tau(q_1) \leq \tau(p-p_k) < \eps/2$. Put $q = q_0 - q_1$. Then $\tau(q) > 1-\eps$. By construction $(p-p_k) \eta q = 0$, so that $\eta q = p_k \eta q$ and hence $\eta q \in K$.
\end{proof}

We now prove several elementary lemmas in preparation for Proposition \ref{prop.form-inv-limit}

\begin{lemma}\label{lem.intersect-rank-dense}
Let $H$ be an $M$-module and $K_n \subset H_n \subset H$ sequences of $M$-submodules. If $K_n$ is rank dense in $H_n$ for all $n$, then $\bigcap_n K_n$ is rank dense in $\bigcap_n H_n$.
\end{lemma}
\begin{proof}
Take $x \in \bigcap_n H_n$ and choose $\eps > 0$. For every $n \in \N$, take a projection $p_n \in M$ with $\tau(p_n) > 1- \eps 2^{-n-1}$ such that $x p_n \in K_n$. Put $p = \bigwedge_n p_n$ and note that $\tau(p) > 1-\eps$. Then $xp \in K_n$ for all $n$, so that $xp \in \bigcap_n K_n$.
\end{proof}

The following lemma is a special case of \cite[Theorem 6.18]{Lu02}, which is stated without proof in \cite{Lu02}. Therefore we provide the details here.

\begin{lemma}\label{lem.decreasing-seq}
Let $H$ be an $M$-module with $\dim_M H < \infty$. Let $K_n \subset H$ be a decreasing sequence of $M$-submodules. Then,
$$\dim_M \Bigl(\bigcap_n K_n\Bigr) = \lim_n \bigl(\dim_M K_n) \; .$$
\end{lemma}
\begin{proof}
Put $K := \bigcap_n K_n$. Note that $\dim_M K_n$ is a decreasing sequence. Denote its limit by $\alpha$. Since $\dim_M K \leq \dim_M K_n$ for all $n$, we have $\dim_M K \leq \alpha$. We need to prove the converse inequality.

We first prove the converse inequality when $H$ is the finitely generated Hilbert $M$-module $p(\C^k \ot L^2(M))$. Then $\clos(K_n) = q_n(\C^k \ot L^2(M))$, where $q_n \in p(\M_k(\C) \ot M)p$ is a decreasing sequence of projections. By Lemma \ref{lem.rank-dense}, we have that $\dim_M K_n = (\Tr \ot \tau)(q_n)$. Denote by $q$ the strong limit of the decreasing sequence of projections $q_n$. Then $\alpha = (\Tr \ot \tau)(q)$. By Lemma \ref{lem.rank-dense}, every $K_n$ is rank dense in $q_n(\C^k \ot L^2(M))$. By Lemma \ref{lem.intersect-rank-dense}, $K$ is rank dense in $q(\C^k \ot L^2(M))$. Hence $\dim_M K = (\Tr \ot \tau)(q) = \alpha$.

We now prove the converse inequality in general. Fix $\eps > 0$. Choose an injective $M$-linear map $\vphi : p(\C^k \ot M) \recht H$ such that $(\Tr \ot \tau)(p) > \dim_M H - \eps$. Denote $H_0 = \image \vphi$. By the rank theorem, we have that $\dim_M(H/H_0) < \eps$. Again by the rank theorem, it follows that $\dim_M(L \cap H_0) > \dim_M L - \eps$ for every $M$-submodule $L \subset H$. Now $\vphi^{-1}(K_n)$ is a decreasing sequence of $M$-submodules of $p(\C^k \ot M)$ with intersection $\vphi^{-1}(K)$. So by the case proven in the previous paragraph, we know that
$$\dim_M (\vphi^{-1}(K)) = \lim_n \dim_M (\vphi^{-1}(K_n)) \; .$$
The left hand side equals $\dim_M(K \cap H_0)$ and the right hand side equals $\dim_M (K_n \cap H_0)$. Since
$$\dim_M(K_n \cap H_0) > \dim_M K_n - \eps \; ,$$
we conclude that $\alpha - \eps \leq \dim_M K$. Since this holds for every $\eps > 0$, we are done.
\end{proof}

\begin{lemma}\label{lem.intersect}
Let $K,H$ be finitely generated Hilbert $M$-modules and $K_n \subset K$ a decreasing sequence of closed $M$-submodules. Let $T : K \recht H$ be a bounded $M$-linear operator. Then $T(\bigcap_n K_n)$ is dense in $\bigcap_n \clos(T(K_n))$.
\end{lemma}
\begin{proof}
Replacing $M$ by matrices over $M$, we may assume that $K = p L^2(M)$, $H = q L^2(M)$ and $T \in q M p$. We have the decreasing sequence of projections $p_n \leq p$ such that $K_n = p_n L^2(M)$. Denote by $p_\infty$ the strong limit of $p_n$ and note that $\bigcap_n K_n = p_\infty L^2(M)$. Define $q_n$ as the left support projection of $T p_n$. Since the sequence $p_n$ is decreasing, also the sequence $q_n$ is decreasing and we denote its limit by $q_\infty$. By construction, $\clos(T(K_n)) = q_n L^2(M)$ and $\bigcap_n \clos(T(K_n)) = q_\infty L^2(M)$. We must prove that the left support projection of $T p_\infty$ equals $q_\infty$. Denote this left support projection by $e$. Clearly $e \leq q_\infty$. Put $f = q_\infty - e$. Since the left support of $T p_n$ equals $q_n$ and since $f \leq q_n$, we have that the left support of $f T p_n$ equals $f$. On the other hand, $f T p_\infty = 0$, implying that $f T (p_n - p_\infty) = f T p_n$. We conclude that the left support of $f T(p_n - p_\infty)$ equals $f$ for all $n$. Hence, $\tau(f) \leq \tau(p_n - p_\infty) \recht 0$. It follows that $f = 0$, so that $e = q_\infty$.
\end{proof}

\subsection{Inverse limits in a weak sense}

\begin{lemma}\label{lem.kind-of-inverse-limit}
Let $H$ be an $M$-module and $H_n \subset H$ a decreasing sequence of $M$-submodules with $\bigcap_n H_n = \{0\}$. Then the sequence $\dim_M (H/H_n)$ is increasing and its limit equals $\dim_M H$.
\end{lemma}
\begin{proof}
Since for $n \geq m$, the natural map $H/H_n \recht H/H_m$ is surjective, it follows from the rank theorem that $\dim_M (H/H_n) \geq \dim_M (H/H_m)$. So $\dim_M (H/H_n)$ is an increasing sequence (in $[0,+\infty]$) and we denote its limit by $r$. Since the natural map $H \recht H/H_n$ is surjective, it also follows from the rank theorem that $\dim_M(H/H_n) \leq \dim_M H$ for al $n$, and hence $r \leq \dim_M H$.

Conversely assume that $p \in \M_k(\C) \ot M$ is a projection and $\vphi : p(\C^k \ot M) \recht H$ is an injective $M$-linear map. It remains to prove that $(\Tr \ot \tau)(p) \leq r$. The $M$-submodules $\vphi^{-1}(H_n)$ form a decreasing sequence whose intersection equals $\{0\}$ by the injectivity of $\vphi$. By Lemma \ref{lem.decreasing-seq}, we get that $\dim_M(\vphi^{-1}(H_n)) \recht 0$. Denote by $\pi_n : H \recht H/H_n$ the quotient map. By the rank theorem,
$$(\Tr \ot \tau)(p) = \dim_M(p(\C^k \ot M)) = \dim_M(\vphi^{-1}(H_n)) + \dim_M(\image(\pi_n \circ \vphi)) \; .$$
Hence, $\dim_M(\image(\pi_n \circ \phi)) \recht (\Tr \ot \tau)(p)$ as $n \recht \infty$. Since
$$\dim_M(\image(\pi_n \circ \vphi)) \leq \dim_M(H/H_n) \leq r$$
for all $n$, we conclude that $(\Tr \ot \tau)(p) \leq r$.
\end{proof}

\begin{lemma}\label{lem.dense-equal-dim}
Let $H$ be a Hilbert $M$-module and $K \subset H$ a dense $M$-submodule. Then $\dim_M H = \dim_M K$.
\end{lemma}
\begin{proof}
Write $H = p(\ell^2(\N) \ovt L^2(M))$ for some projection $p \in \B(\ell^2(\N)) \ovt M$. Choose an increasing sequence of projections $p_n \in \B(\ell^2(\N)) \ovt M$ such that $p_n \leq p$, $p_n \recht p$ strongly and such that for every $n$, the center valued trace of $p_n$ is bounded. This means that $H_n := p_n(\ell^2(\N) \ovt  L^2(M))$ is a finitely generated Hilbert $M$-module for every $n$. Write $\vphi_n : H \recht H_n : \vphi_n(\xi) = p_n \xi$. Then $K \cap \Ker \vphi_n$ is a decreasing sequence of $M$-submodules of $K$ with trivial intersection. By Lemma \ref{lem.kind-of-inverse-limit}, we get that $\dim_M K = \lim_n \dim_M \vphi_n(K)$. Since $K$ is dense in $H$, we get that $\vphi_n(K)$ is dense in $H_n$. By Lemma \ref{lem.rank-dense},

we get that $\vphi_n(K) \subset H_n$ is rank dense. Hence,
$$\dim_M \vphi_n(K) = \dim_M H_n = (\Tr \ot \tau)(p_n) \recht (\Tr \ot \tau)(p) = \dim_M H \; .$$
So we have proven that $\dim_M K = \dim_M H$.
\end{proof}

\begin{lemma}\label{lem.again-kind-of-inverse-limit}
Let $H$ be an $M$-module and $H_n$ a sequence of Hilbert $M$-modules. Let $\vphi_n : H \recht H_n$ be $M$-linear maps such that $\Ker \vphi_n$ is a decreasing sequence of $M$-submodules of $H$ with $\bigcap_n \Ker \vphi_n = \{0\}$. Then,
$$\dim_M H = \lim_n \dim_M (\clos(\vphi_n(H))) \; .$$
\end{lemma}
\begin{proof}
The lemma is an immediate consequence of Lemmas \ref{lem.kind-of-inverse-limit} and Lemma \ref{lem.dense-equal-dim}.
\end{proof}

\subsection{Inverse limits in a strong sense}

An \emph{inverse system} of $M$-modules consists of a sequence of $M$-modules $H_n$ and, for every $k \geq n$, an $M$-linear map $\pi_{n,k} : H_k \recht H_n$ such that $\pi_{n,k} \circ \pi_{k,m} = \pi_{n,m}$. The \emph{inverse limit} $\invlimit H_n$ of the inverse system is the $M$-module $\cH$ consisting of all sequences $(x_n)$ with $x_n \in H_n$ for all $n$ and $\pi_{n,k}(x_k) = x_n$ for all $k \geq n$. We denote by $\pi_n : \cH \recht H_n : (x_n) \mapsto x_n$ the natural $M$-linear map from $\cH$ to $H_n$.

An inverse system of Hilbert $M$-modules is an inverse system in which all the $H_n$ are Hilbert $M$-modules and where the $\pi_{n,k}$ are bounded $M$-linear operators. Then $\invlimit H_n$ naturally is a Fr\'{e}chet $M$-module.

We need the following result from \cite{CG85}. For this result to be true, it is crucial that the Hilbert $M$-modules $H_n$ are finitely generated.

\begin{proposition}\label{prop.inv-limit-CG}
Let $\pi_{n,k} : H_k \recht H_n$ be an inverse system of \emph{finitely generated} Hilbert $M$-modules with inverse limit $\cH = \invlimit H_n$. Then
\begin{enumerate}
\item (\cite[Lemma 2.1]{CG85}) $\dis\quad \clos(\pi_n(\cH)) = \bigcap_{k \geq n} \clos(\pi_{n,k}(H_k)) \; .$
\item (\cite[Theorem 6.18]{Lu02}) $\dis\quad \dim_M \cH = \lim_{n \recht \infty} \bigl(\lim_{k \recht \infty} \dim_M(\clos(\pi_{n,k}(H_k)))\bigr) \; .$
\end{enumerate}
\end{proposition}
\begin{proof}
The first statement is exactly \cite[Lemma 2.1]{CG85}. In combination with Lemmas \ref{lem.decreasing-seq} and \ref{lem.again-kind-of-inverse-limit}, it also implies the second statement.
\end{proof}

The following result can be deduced from the self-injectivity of the algebra $\cM$ of operators affiliated with $M$ and from the fact that dualizing $\cM$-modules preserves the dimension (see {\cite[Corollary 3.4]{Th06}}). For the convenience of the reader, we show how to deduce the result from the above more elementary results.

\begin{proposition}\label{prop.form-inv-limit}
Let $\rho_{n,k} : K_k \recht K_n$ and $\pi_{n,k} : H_k \recht H_n$ be two inverse systems of \emph{finitely generated} Hilbert $M$-modules with inverse limits
$$\cK = \invlimit K_n \quad\text{and}\quad \cH = \invlimit H_n \; .$$
Assume that $T_n : K_n \recht H_n$ is a sequence of bounded $M$-linear operators satisfying $T_n \circ \rho_{n,k} = \pi_{n,k} \circ T_k$ for all $k \geq n$. Denote by $T : \cK \recht \cH$ the unique continuous $M$-linear operator that satisfies $\pi_n \circ T = T_n \circ \rho_n$ for all $n$. Then the natural $M$-linear maps
$$\frac{\cH}{T(\cK)} \;\;\longrightarrow\;\; \frac{\cH}{\clos(T(\cK))} \;\;\longrightarrow\;\; \invlimit \frac{H_n}{\clos(T_n(K_n))}$$
are isomorphisms in rank.
\end{proposition}

\begin{proof}

{\bf Claim 1.} {\it Assume that $\dim_M H_n \leq 1$ and that $T_n(K_n)$ is dense in $H_n$ for all $n$. Then $T(\cK)$ is rank dense in $\cH$.}

Fix $n \in \N$. By Proposition \ref{prop.inv-limit-CG}, we have that
$$\clos(\rho_n(\cK)) = \bigcap_{k \geq n} \clos(\rho_{n,k}(K_k)) \; .$$
Applying $T_n$ and using Lemma \ref{lem.intersect}, we get that
\begin{equation}\label{eq.111}
\clos(T_n(\rho_n(\cK))) = \bigcap_{k \geq n} \clos(T_n(\rho_{n,k}(K_k))) \; .
\end{equation}
The left hand side of \eqref{eq.111} equals $\clos(\pi_n(T(\cK)))$. Using the density of $T_k(K_k)$ in $H_k$ and using again Proposition \ref{prop.inv-limit-CG},
the right hand side of \eqref{eq.111} equals
$$\bigcap_{k\geq n} \clos(\pi_{n,k}(T_k(K_k))) = \bigcap_{k\geq n} \clos(\pi_{n,k}(H_k)) = \clos(\pi_n(\cH)) \; .$$
So we conclude that
$$\clos(\pi_n(T(\cK))) = \clos(\pi_n(\cH)) \quad\text{for all}\;\; n \in \N \; .$$
Using Lemma \ref{lem.rank-dense}, we get that
$$\dim_M(\pi_n(T(\cK))) = \dim_M(\pi_n(\cH)) \quad\text{for all}\;\; n \in \N \; .$$
We let $n$ tend to infinity. By Lemma \ref{lem.kind-of-inverse-limit}, the left hand side converges to $\dim_M(T(\cK))$, while the right hand side converges to $\dim_M(\cH)$ and remains bounded by $1$. So we get that
$$\dim_M(T(\cK)) = \dim_M(\cH) \leq 1 \; .$$
It follows from Theorem \ref{thm.basic} that the quotient $\cH / T(\cK)$ has dimension zero, meaning that $T(\cK)$ is rank dense in $\cH$. So the first claim is proven.

{\bf Claim 2.} {\it If $T_n(K_n)$ is dense in $H_n$ for all $n$, then $T(\cK)$ is rank dense in $\cH$.}

Fix $x \in \cH$ and $\eps > 0$. Define $H_n' := \clos(\pi_n(x)  M)$ and $\cH' = \invlimit H_n'$. We view $\cH'$ as an $M$-submodule of $\cH$. Put $K'_n := T_n^{-1}(H_n')$ and $\cK' = \invlimit K'_n$. We also view $\cK'$ as an $M$-submodule of $\cK$. By the first claim, $T(\cK')$ is rank dense in $\cH'$. Since $x \in \cH'$, we find a projection $p \in M$ with $\tau(p) > 1-\eps$ and $x p \in T(\cK')$. So certainly $x p \in T(\cK)$ and the second claim is proven.

{\bf Proof of the proposition.} Define $L_n := \clos(T_n(K_n))$ and view the inverse limit $\cL:=\invlimit L_n$ as a closed $M$-submodule of $\cH$. By construction, $T_n(K_n)$ is dense in $L_n$. So by claim~2, we get that $T(\cK)$ is rank dense in $\cL$. It follows that
$$\frac{\cH}{T(\cK)} \;\;\longrightarrow\;\; \frac{\cH}{\clos(T(\cK))} \;\;\longrightarrow\;\; \frac{\cH}{\cL}$$
are rank isomorphisms. From claim~2, it also follows that the natural map $\cH \recht \invlimit H_n / L_n$ has a rank dense image. Its kernel is by construction equal to $\cL$, so that the natural map
$$\frac{\cH}{\cL} \;\;\longrightarrow\;\; \invlimit \frac{H_n}{L_n}$$
is a rank isomorphism.

\end{proof}

\subsection{Dimension theory for semifinite von Neumann algebras}

The Murray-von Neumann dimension of arbitrary (purely algebraic) modules over a tracial von Neumann algebra $(M,\tau)$ was defined in \cite{Lu97}. This was extended to the case of semifinite von Neumann algebras $(N,\Tr)$ in \cite[Appendix B]{Pe11}. We give here a more direct approach to the results of \cite{Pe11}.

We define the $N$-dimension of an arbitrary $N$-module over a von Neumann algebra $N$ equipped with a normal semifinite faithful trace $\Tr$. In doing so, we systematically make use of the dimension of $pNp$-modules, where $p \in N$ is a projection with $\Tr(p) < \infty$. We always implicitly equip $pNp$ with the faithful normal tracial state $\tau(x) = \Tr(p)^{-1} \Tr(x)$.

\begin{definition}\label{def.dim-semifinite}
Let $(N,\Tr)$ be a von Neumann algebra with separable predual equipped with a normal semifinite faithful trace. For every $N$-module $H$, we define
$$\dim_N H := \sup \{\Tr(p) \dim_{p N p} (H p) \mid p \in N \;\;\text{a projection with}\;\; \Tr(p) < \infty \} \; .$$
\end{definition}

Definition \ref{def.dim-semifinite} is motivated by the following easy lemma, generalizing a fact noted in the proof of \cite[Theorem 2.4]{CS04}.

\begin{lemma}\label{lem.dim-reduced}
Let $(M,\tau)$ be a von Neumann algebra with separable predual equipped with a faithful normal tracial state. Let $p \in M$ be a projection with central support $z \in \cZ(M)$. Then for every $M$-module $H$, we have
$$\tau(p) \, \dim_{pMp}(H p) = \dim_M (Hz) \; .$$
\end{lemma}
\begin{proof}
Denote by $\tau_p$ the faithful normal tracial state on $pMp$ given by $\tau_p(x) = \tau(p)^{-1} \tau(x)$. As explained in the beginning of this section, we always consider $\dim_{pMp}$ with respect to this tracial state.

First assume that $q \in \M_k(\C) \ot M$ is a projection and that $\vphi : q(\C^k \ot M) \recht Hz$ is an injective $M$-linear map. Note that the injectivity of $\vphi$ forces $q \leq 1 \ot z$. The restriction of $\vphi$ to $q(\C^k \ot Mp)$ is an injective $pMp$-linear map into $Hp$. Hence,
$$\dim_{pMp} (Hp) \geq \dim_{pMp} (q(\C^k \ot Mp)) \; .$$
Since $q \leq 1 \ot z$, the right hand side equals $\tau(p)^{-1} \, (\Tr \ot \tau)(q)$. Since this holds for all injective $M$-linear maps $\vphi$, we get that
$$\dim_{pMp} (Hp) \geq \tau(p)^{-1} \, \dim_M (Hz) \; .$$

Conversely assume that $q \in \M_k(\C) \ot pMp$ is a projection and that $\vphi : q(\M_{k,1}(\C) \ot pMp) \recht Hp$ is an injective $pMp$-linear map. Define $\xi \in \M_{1,k}(\C) \ot Hp$ given by
$$\xi := \sum_{i=1}^k e_{1i} \ot \vphi(q(e_{i1} \ot p)) \; .$$
Note that $\xi q = \xi$ and $\vphi(\eta) = \xi \eta$ for all $\eta \in q(\M_{k,1}(\C) \ot pMp)$. Define
$$\psi : q(\M_{k,1}(\C) \ot M) \recht H : \psi(\eta) = \xi \eta \; .$$
Observe that $\psi$ takes values in $Hz$ and that $\psi$ is an injective $M$-linear map. Hence,
$$(\Tr \ot \tau)(q) \leq \dim_M (Hz) \; .$$
The left hand side equals $\tau(p) \, (\Tr \ot \tau_p)(q)$. Since this holds for all injective $pMp$-linear maps $\vphi$, we get that
$$\tau(p) \, \dim_{pMp} (Hp) \leq \dim_M (Hz) \; .$$
\end{proof}

\begin{lemma}\label{lem.realize-semifinite-dim}
Let $(N,\Tr)$ be a von Neumann algebra with separable predual equipped with a normal semifinite faithful trace. Let $H$ be an $N$-module. If $p_n$ is an increasing sequence of projections in $N$ with $\Tr(p_n) < \infty$ and such that the central supports $z_n$ of $p_n$ converge strongly to $1$, then the sequence
$$\Tr(p_n) \, \dim_{p_n N p_n}(H p_n) \quad\text{is increasing and converges to}\quad \dim_N H \; .$$
In particular, if $p \in N$ is a projection of finite trace and central support equal to $1$, we have $\dim_N H = \Tr(p) \, \dim_{pNp} (Hp)$.
\end{lemma}
\begin{proof}
If $p \leq p'$ are projections in $N$ with $\Tr(p) \leq \Tr(p') < \infty$, we apply Lemma \ref{lem.dim-reduced} to the von Neumann algebra $p' N p'$. Denoting by $z$ the central support of $p$ in $N$, we conclude that
$$\Tr(p) \, \dim_{pNp} (H p) = \Tr(p') \, \dim_{p' N p'}(Hp' z) \leq \Tr(p') \, \dim_{p' N p'}(H p') \; .$$
So the sequence in the formulation of the lemma is indeed increasing. Denote its limit by $\alpha \in [0,+\infty]$.

By construction, $\alpha \leq \dim_N H$. To prove the converse inequality, choose an arbitrary projection $q \in N$ with $\Tr(q) < \infty$. We must prove that $\Tr(q) \, \dim_{qNq} (Hq) \leq \alpha$. Denote by $z$ the central support of $q$ in $N$. Put $e_n := q \vee p_n$. Note that $\Tr(e_n) < \infty$. Inside $e_n N e_n$, the central supports of $p_n$ and $q$ are respectively equal to $e_n z_n$ and $e_n z$. Applying twice Lemma \ref{lem.dim-reduced}, it follows that
$$\Tr(p_n) \, \dim_{p_n N p_n}(H z p_n) = \Tr(e_n) \, \dim_{e_n N e_n} (H z z_n e_n) = \Tr(q) \, \dim_{qNq} (H z_n q) \; .$$
Since $H z p_n \subset H p_n$, the left hand side is smaller or equal than $\al$. It remains to show that the right hand side converges to $\Tr(q) \, \dim_{qNq}(H q)$.

Put $H_0 = \{\xi \in H q \mid \xi z_n = 0 \;\text{for all}\; n \in \N\}$. Since $z_n$ is an increasing sequence of projections that strongly converges to $1$, we have by construction that $H_0$ is a $qNq$-module of rank zero. Hence, $\dim_{qNq}(H q) = \dim_{qNq}(Hq / H_0)$. Using the surjective $qNq$-linear maps
$$\frac{H q}{H_0} \recht H z_n q : \xi \mapsto \xi z_n \; ,$$
it follows from Lemma \ref{lem.kind-of-inverse-limit} that $\dim_{qNq}(H z_n q) \recht \dim_{qNq}(Hq / H_0)$.
\end{proof}

It is now easy to prove the semifinite version of Lemma \ref{lem.again-kind-of-inverse-limit}.

\begin{lemma}\label{lem.weak-inverse-limit-semifinite}
Let $(N,\Tr)$ be a von Neumann algebra with separable predual equipped with a normal semifinite faithful trace. Let $H$ be an $N$-module. Assume that $H_n$ is a sequence of Hilbert $N$-modules and that $\vphi_n : H \recht H_n$ are $N$-linear maps such that $\Ker \vphi_n$ is a decreasing sequence of $N$-submodules of $H$ with $\bigcap_n \Ker \vphi_n = \{0\}$. Then
$$\dim_N H = \lim_n \dim_N \clos(\vphi_n(H)) \; .$$
\end{lemma}
\begin{proof}
Choose an increasing sequence of projections $p_k \in N$ such that $\Tr(p_k) < \infty$ for all $k$ and such that the central supports $z_k$ of $p_k$ converge strongly to $1$. Consider the double sequence $\al_{n,k} := \Tr(p_k) \dim_{p_k N p_k} (\clos(\vphi_n(H)) p_k)$. For fixed $n$ and increasing $k$, by Lemma \ref{lem.realize-semifinite-dim}, the sequence $\al_{n,k}$ is increasing and converges to $\dim_N \clos(\vphi_n(H))$.

For fixed $k$ and increasing $n$, we apply Lemma \ref{lem.again-kind-of-inverse-limit} to $p_k N p_k$ and the restriction of $\vphi_n$ to $H p_k$, and conclude that $\al_{n,k}$ is increasing to the limit $\Tr(p_k) \dim_{p_k N p_k} H p_k$. When $k \recht \infty$, this last sequence increases to $\dim_N H$ by Lemma \ref{lem.realize-semifinite-dim}. In combination with the previous paragraph, the lemma is proven.
\end{proof}

\renewcommand{\thesection}{B}
\section{Properties of cross section equivalence relations}\label{appB}

In this section we prove the ``folklore'' Proposition \ref{prop.properties-cross-section-eq-rel}. We do not claim any originality, but in order to keep our article as clear and self-contained as possible, we give a detailed argument. The construction of the invariant probability measure $\nu$ on the cross section equivalence relation $\cR$ is a very special case of Connes's transverse measure theory, see \cite{Co79} and \cite[Appendix A.1]{ADR00}. In particular, point \ref{p7} of Proposition \ref{prop.properties-cross-section-eq-rel} is a very special case of \cite[Theorem 3.2.16]{ADR00}. Nevertheless we think that the following explicit and direct approach is useful.

We need to introduce a bit of terminology from the theory of countable equivalence relations.

Let $\cR$ be a countable pmp equivalence relation on the standard probability space $(Y,\nu)$. A (right) Borel action of $\cR$ on a standard Borel space $Z$ consists of Borel maps $\pi \colon Z \recht Y$ and
\begin{equation}\label{eq.cZ}
\al \colon \cZ \recht Z \quad\text{where}\quad \cZ = \{(z,y) \in Z \times Y \mid (\pi(z),y) \in \cR\}
\end{equation}
satisfying $\pi(\al(z,y)) = y$, $\al(z,\pi(z))=z$ and $\al(\al(z,y),y') = \al(z,y')$ whenever $(z,y) \in \cZ$ and $(y,y') \in \cR$. For every $\psi \in [[\cR]]$ and every $z \in Z$ with $\pi(z) \in R(\psi)$, we denote $z \cdot \psi := \al(z,\psi^{-1}(\pi(z)))$. In this way, $[[\cR]]$ acts on the right on $Z$.

If we are moreover given a $\sigma$-finite measure $\eta$ on $Z$, we say that the action is \emph{nonsingular} if $\eta(\pi^{-1}(A)) = 0$ whenever $\nu(A) = 0$ and if for every $\psi \in [[\cR]]$, the partial bijection $z \mapsto z \cdot \psi$ is nonsingular. We then have a right action of $[[\cR]]$ on $L^\infty(Z)$ given by
$$(a \cdot \psi)(z) = a(z \cdot \psi^{-1}) \quad\text{for all}\quad a \in L^\infty(Z), \psi \in [[\cR]], z \in Z \; .$$

We say that a nonsingular automorphism $\delta$ of $(Z,\eta)$ \emph{commutes} with the action of $\cR$ on $(Z,\eta)$ if $\pi(\delta(z)) = \pi(z)$ for all $z \in Z$ and if $\al(\delta(z),y) = \delta(\al(z,y))$ for all $(z,y) \in \cZ$.

Following \cite[Definition 6]{CFW81}, we say that $\cR$ is \emph{amenable} if there exists a (typically non-normal) conditional expectation $P \colon L^\infty(\cR) \recht L^\infty(Y)$ satisfying $P(\psi \cdot f) = \psi \cdot P(f)$ for all $f \in L^\infty(\cR)$ and all $\psi \in [[\cR]]$. Here we used the natural left actions of $\cR$ on $Y$ and on $\cR$. We call $P$ a left invariant mean on $\cR$.

\begin{lemma}\label{lem.project-fixed-points}
Let $\cR$ be an amenable countable pmp equivalence relation on the standard probability space $(Y,\nu)$. Assume that we are given a nonsingular action of $\cR$ on the standard measure space $(Z,\eta)$. Denote by $L^\infty(Z)^\cR$ the von Neumann subalgebra of $L^\infty(Z)$ consisting of the $\cR$-invariant bounded measurable functions. Then there exists a (typically non-normal) conditional expectation
$$Q \colon L^\infty(Z) \recht L^\infty(Z)^\cR$$
satisfying $\delta_* \circ Q = Q \circ \delta_*$ for every nonsingular automorphism $\delta$ of $(Z,\eta)$ that commutes with the action of $\cR$.
\end{lemma}
\begin{proof}
Define $\cZ$ as in \eqref{eq.cZ} and equip $\cZ$ with the $\sigma$-finite measure given by integrating w.r.t.\ $\eta$ the counting measure over the map $\cZ \recht Z \colon (z,y) \mapsto z$. Every normal conditional expectation $E \colon L^\infty(Z) \recht L^\infty(Y)$ uniquely extends to a normal conditional expectation $\cE \colon L^\infty(\cZ) \recht L^\infty(\cR)$. Fix a left invariant mean $P \colon L^\infty(\cR) \recht L^\infty(Y)$. We claim that there is a unique conditional expectation
\begin{equation}\label{eq.char-cP}
\cP \colon L^\infty(\cZ) \recht L^\infty(Z) \quad\text{satisfying}\quad E \circ \cP = P \circ \cE
\end{equation}
for every normal conditional expectation $E \colon L^\infty(Z) \recht L^\infty(Y)$. To prove this claim, first fix a faithful normal conditional expectation $E_0 \colon L^\infty(Z) \recht L^\infty(Y)$. Denote by $\tau$ the state on $L^\infty(Y)$ given by integration w.r.t.\ $\nu$. Denote $\tautil = \tau \circ E_0$. Using the Cauchy-Schwarz inequality, we find a unique conditional expectation $\cP \colon L^\infty(\cZ) \recht L^\infty( Z)$ such that
$$
\tautil(b^* \cP(a) c) = (\tau \circ P \circ \cE_0)(b^* a c) \quad\text{for all}\;\; a \in L^\infty(\cZ) \; , \; b,c \in L^\infty(Z) \; .
$$
By construction, $E_0 \circ \cP = P \circ \cE_0$. If $E$ is another normal conditional expectation and $E \leq 2 E_0$, there is a positive $a \in L^\infty(Z)$ with $\|a\|\leq 2$ and $E(f) = E_0(af)$ for all $f \in L^\infty(Z)$. It follows that $\cE(f) = \cE_0(a f)$ for all $f \in L^\infty(\cZ)$ and hence $E \circ \cP = P \circ \cE$. If $E$ is an arbitrary normal conditional expectation, put $E_1 = (E_0 + E)/2$. Then $E_1$ is a faithful normal conditional expectation and we find a unique $\cP_1$ satisfying $E_1 \circ \cP_1 = P \circ \cE_1$. Since $E_0 \leq 2 E_1$, we also have $E_0 \circ \cP_1 = P \circ \cE_0$. Hence $\cP_1 = \cP$. Since $E \leq 2 E_1$, we have $E \circ \cP_1 = P \circ \cE$. Since $\cP_1 = \cP$, we have proven that $\cP$ satisfies \eqref{eq.char-cP} for every normal conditional expectation $E$.

Considering the right action of $[[\cR]]$ on $L^\infty(\cZ)$ given by
$$(a \cdot \psi)(z,y) = a(z \cdot \psi^{-1},y) \; ,$$
we claim that $\cP(a \cdot \psi) = \cP(a) \cdot \psi$ for all $a \in L^\infty(\cZ)$ and $\psi \in [[\cR]]$. Since $\cP$ is $L^\infty(Y)$-linear, it suffices to prove this formula for $\psi \in [\cR]$. Whenever $E \colon L^\infty(Z) \recht L^\infty(Y)$ is a normal conditional expectation and $\psi \in [\cR]$, also $E_\psi(a) = E(a \cdot \psi^{-1}) \cdot \psi$ is a normal conditional expectation.

We similarly define $\cP_\psi$ by the formula $\cP_\psi(a) = \cP(a \cdot \psi^{-1}) \cdot \psi$. Since the unique normal conditional expectation of $L^\infty(\cZ)$ onto $L^\infty(\cR)$ extending $E_\psi$ is given by the formula
$$\cE_\psi(a) = \psi^{-1} \cdot \cE(a \cdot \psi^{-1}) \; ,$$
we get that
\begin{align*}
E(\cP_\psi(a)) &= E(\cP(a \cdot \psi^{-1}) \cdot \psi) = \bigl((E_{\psi^{-1}} \circ \cP)(a \cdot \psi^{-1})\bigr) \cdot \psi \\
&= \bigl((P \circ \cE_{\psi^{-1}})(a \cdot \psi^{-1})\bigr) \cdot \psi = P(\psi \cdot \cE(a)) \cdot \psi = \psi^{-1} \cdot P(\psi \cdot \cE(a)) = P(\cE(a)) \; .
\end{align*}
So we get that $E \circ \cP_\psi = P \circ \cE$ for all normal conditional expectations $E\colon L^\infty(Z) \recht L^\infty(Y)$. Hence $\cP_\psi = \cP$ and the claim is proven.

We define the conditional expectation $Q \colon L^\infty(Z) \recht L^\infty(Z)^\cR$ given by
\begin{equation}\label{eq.mymeanQ}
Q(a) = \cP(b) \quad\text{where}\quad b(z,y) = a(\al(z,y)) \; .
\end{equation}
Since $b \cdot \psi = b$ for all $\psi \in [[\cR]]$, we indeed have that $Q(a) \in L^\infty(Z)^\cR$.

Finally take a nonsingular automorphism $\delta$ of $(Z,\eta)$ that commutes with the action of $\cR$. The ``functoriality'' of the above construction of $Q$ ensures that $\delta_* \circ Q = Q \circ \delta_*$.
\end{proof}

\subsection{Proof of Proposition \ref{prop.properties-cross-section-eq-rel}}

{\bf Proof of \ref{p1}.} Fix a cross section $Y \subset X$ and a neighborhood $\cU \subset G$ of $e$ such that $\theta \colon \cU \times Y \recht X : (g,y) \mapsto g \cdot y$ is injective. Write $\cW := \cU \cdot Y$. Define $\cR$ as in \ref{p1}. Then $\cR$ is an equivalence relation on $Y$. Since $\theta$ is injective, the map $G \times Y \recht X : (g,y) \mapsto g \cdot y$ is countable-to-one. Define the Borel maps $\Psi \colon G \times Y \recht X \times Y$ and $\pi_\ell \colon X \times Y \recht X$ as in \ref{p2}.
Since $\pi_\ell \circ \Psi$ is countable-to-one, also $\Psi$ is countable-to-one. So $Z = \image \Psi$ is a Borel set. Then also $\cR = Z \cap (Y \times Y)$ is a Borel set, meaning that $\cR$ is a Borel equivalence relation. We get as well that $\cR \recht Y : (y,y') \mapsto y$ is countable-to-one. So $\cR$ is a countable Borel equivalence relation on $Y$. This proves \ref{p1}, as well as the facts that $Z$ is a Borel set and that $\pi_\ell \colon Z \recht X$ is countable-to-one.

{\bf Proof of \ref{p2}.} To prove the other statements of the proposition, we may discard a conegligible $G$-invariant Borel subset of $X$ and assume that $G$ acts freely on $X$ and that $G \cdot Y = X$. Define the $\sigma$-finite measure $\eta$ on $Z$ as in \ref{p2}. Since $\mu$ is invariant under $G \actson X$, the measure $\eta$ is invariant under the action $G \actson Z$ given by $g \cdot (x,y) = (g \cdot x,y)$. Since the action $G \actson X$ is free, $\Psi \colon G \times X \recht Z$ is a bijection. So, $(\Psi^{-1})_*(\eta)$ is a $G$-invariant measure on $G \times Y$. By the uniqueness of the Haar measure, there exists a unique $\sigma$-finite measure $\nu_1$ on $Y$ such that $\Psi_*(\lambda \times \nu_1) = \eta$. Since $\pi_\ell \circ \Psi$ is injective on $\cU \times Y$, we get that $\lambda(\cU) \, \nu_1(Y) = \mu(\cU \cdot Y)$. In particular, $\lambda(\cU)$ and $\nu_1(Y)$ are finite. Putting $\covol Y := \lambda(\cU) / \mu(\cU \cdot Y)$ and $\nu := \covol Y  \cdot \nu_1$, we have proven \ref{p2}.

{\bf Proof of \ref{p3} and \ref{p4}.} Take another cross section $Y' \subset X$ (and for the proof of \ref{p3}, we will later take $Y' = Y$). Define the equivalence relation $\cR'$, the probability measure $\nu'$ and the Borel set $Z'$ as in \ref{p2}.
Define
$$\cS := \{(y,y') \in Y \times Y' \mid y \in G \cdot y' \} \; .$$
Since $\cS = Z' \cap (Y \times Y')$, we get that $\cS$ is a Borel set. We denote by $\pi_\ell \colon \cS \recht Y$ and $\pi_r \colon \cS \recht Y'$ the projections on the first, resp.\ second coordinate. Both projections are countable-to-one and we define the $\sigma$-finite measure $\gamma_\ell$ on $\cS$ by integrating w.r.t.\ $\nu$ the counting measure over the map $\pi_\ell$. We similarly define $\gamma_r$ by integrating w.r.t.\ $\nu'$ the counting measure over $\pi_r$. We claim that
\begin{equation}\label{eq.proportion}
\covol(Y)^{-1} \cdot \gamma_\ell = \covol(Y')^{-1} \cdot \gamma_r \; .
\end{equation}

To prove this claim, we define
\begin{align*}
& \cZ := \{(x,y,y') \in X \times Y \times Y' \mid G \cdot x = G \cdot y = G \cdot y' \} \quad , \\
& \Phi_1 \colon G \times \cS \recht \cZ : (g,y,y') = (g \cdot y, y, y') \quad\text{and}\\
& \Phi_2 \colon G \times \cS \recht \cZ : (g,y,y') = (g \cdot y', y,y') \; .
\end{align*}
Denote by $\pi_\ell \colon \cZ \recht X$ the projection on the first coordinate. Denote by $\rho$ the $\sigma$-finite measure on $\cZ$ given by integrating w.r.t.\ $\mu$ the counting measure over $\pi_\ell$. Using the intermediate projection $\cZ \recht Z \recht X$ on the first two coordinates and using the fact that $\Psi_*(\lambda \times \nu) = \covol Y \cdot \eta$, we get that
\begin{equation}\label{eq.firstpush}
(\Phi_1)_*(\lambda \times \gamma_\ell) = \covol Y \cdot \rho \; .
\end{equation}
Similarly using the intermediate projection on the first and third coordinate, we get that
\begin{equation}\label{eq.secondpush}
(\Phi_2)_*(\lambda \times \gamma_r) = \covol Y' \cdot \rho \; .
\end{equation}
Since $G$ acts freely on $X$, we can uniquely define the Borel map $\Omega \colon \cS \recht G$ satisfying $$\Omega(y,y') \cdot y' = y \quad\text{for all}\quad (y,y') \in \cS \; .$$
We then note that $\Phi_1 = \Phi_2 \circ \zeta$ where $\zeta(g,y,y') = (g \Omega(y,y'),y,y')$. Since $G$ is unimodular, the map $\zeta$ preserves the measure $\lambda \times \gamma_\ell$. So \eqref{eq.firstpush} and \eqref{eq.secondpush} imply that \eqref{eq.proportion} holds.

{\bf End of the proof of \ref{p3}.} In the particular case where $Y = Y'$, equation \eqref{eq.proportion} precisely says that $\cR$ preserves $\nu$. So \ref{p3} is proven.

{\bf End of the proof of \ref{p4}.} Fix an arbitrary nonnegligible $G$-invariant Borel subset $X_1 \subset X$. Put $Y_1 = X_1 \cap Y$, $Y_1' = X_1 \cap Y'$ and $\cS_1 := \cS \cap (X_1 \times X_1)$. By \ref{p2}, we have $\nu(Y_1) > 0$ and $\nu'(Y'_1) > 0$. So $\gamma_\ell(\cS_1) > 0$. Since both $\pi_\ell \colon \cS_1 \recht Y_1$ and $\pi_r \colon \cS_1 \recht Y'_1$ are countable-to-one and surjective, we can choose a Borel subset $\cS_2 \subset \cS_1$ with $\gamma_\ell(\cS_2) > 0$ such that the restrictions of $\pi_\ell$ and $\pi_r$ to $\cS_2$ are injective. We put $Y_2 := \pi_\ell(\cS_2)$ and $Y'_2 := \pi_r(\cS_1)$. We define the bijective Borel map $\al_2 \colon Y_2 \recht Y'_2$ such that $\al_2 \circ \pi_\ell = \pi_r$ on $\cS_2$. By \eqref{eq.proportion}, we have
$$(\al_2)_*(\nu_{|Y_2}) = \frac{\covol Y}{\covol Y'} \, \nu'_{|Y'_2} \; .$$
By construction, $\al_2$ is an isomorphism between the restricted equivalence relations $\cR_{|Y_2}$ and $\cR'_{|Y'_2}$.

Also note that by \ref{p2}, the set $X_2 = G \cdot Y_2 = G \cdot Y'_2$ is a nonnegligible $G$-invariant Borel subset of $X_1$. Therefore by a maximality argument, we can find a Borel subset $\cS_0 \subset \cS$ such that the restrictions of $\pi_\ell$ and $\pi_r$ to $\cS_0$ are injective and their respective images $Y_0$ and $Y_0'$
satisfy
$$G \cdot Y_0 = G \cdot Y'_0 \quad\text{and}\quad \mu(X - G \cdot Y_0) = \mu(X - G \cdot Y'_0) = 0 \; .$$
We define the bijective Borel map $\al \colon Y_0 \recht Y_0'$ such that $\al \circ \pi_\ell = \pi_r$. Then $\al$ satisfies all the conditions in \ref{p4}.

{\bf Proof of \ref{p5}.} Since the map $G \times Y \recht X : (g,y) \mapsto g \cdot y$ is countable-to-one and surjective, it admits a Borel right inverse $x \mapsto (\vphi(x),\pi(x))$. Then note that a Borel map $F \colon X \recht \C$ is $G$-invariant if and only if it is of the form $F_0 \circ \pi$ for some $\cR$-invariant Borel map $F_0 \colon Y \recht \C$. From this, \ref{p5} follows immediately.

{\bf Proof of \ref{p6}.} If $G$ is compact, one checks that $\cR$ has finite orbits. Conversely assume that $Y_0 \subset Y$ is a nonnegligible subset such that every $y \in Y_0$ has a finite orbit under $\cR$. Choose a fundamental domain $Y_1$ for the equivalence relation $\cR \cap (Y_0 \times Y_0)$. So $Y_1 \subset Y$ is nonnegligible and has the following property: if $y,y' \in Y_1$ and $(y,y') \in \cR$, then $y=y'$. Choose a compact neighborhood $K$ of $e$ such that $K \subset \cU$.
Define $\cW_1 := K \cdot Y_1$. From \eqref{eq.another-push}, we know that $\cW_1$ is a nonnegligible subset of $X$. We prove the following claim: if $g \in G$ and if $g \cdot \cW_1 \cap \cW_1 \neq \emptyset$, then $g \in K K^{-1}$. To prove this claim, assume that $g \in G$, $x \in \cW_1$ and $g \cdot x \in \cW_1$. Take $h,h' \in K$ and $y,y' \in Y_1$ such that $x = h \cdot y$ and $g \cdot x = h' \cdot y'$. It follows that $(y,y') \in \cR$ and hence $y=y'$. Since $G$ acts freely on $X$, it then follows that $gh = h'$, so that indeed $g \in K K^{-1}$.

Since $\mu$ is a probability measure and since $\cW_1 \subset X$ is nonnegligible, we can take a finite sequence of elements $g_1,\ldots,g_n \in G$ that is maximal with respect to the property that the sets $(g_k \cdot \cW_1)_{k=1,\ldots,n}$ are disjoint. Using the claim in the previous paragraph, it follows that
$$G = \bigcup_{k=1}^n g_k K K^{-1}$$
so that $G$ is compact. This ends the proof of \ref{p6}.

{\bf Proof of \ref{p7}.} Consider as above the measure space $(Z,\eta)$, together with the left action of $G$ given by $g \cdot (x,y) = (g \cdot x, y)$ and the right action of $\cR$ given by $\al(x,y,y') = (x,y')$ for all $(x,y,y') \in \cZ$ where
$$\cZ = \{(x,y,y') \in X \times Y \times Y \mid G \cdot x = G \cdot y = G \cdot y'\} \; .$$
Note that these actions commute. First assume that $G$ is an amenable lcsc group. Integrating over an invariant mean on $G$, we obtain a (non-normal) conditional expectation $Q \colon L^\infty(Z) \recht L^\infty(Z)^G$ satisfying $Q(a \cdot \psi) = Q(a) \cdot \psi$ for all $a \in L^\infty(Z)$ and all $\psi \in [[\cR]]$. Using \eqref{eq.first-push}, it follows that $L^\infty(Z)^G = L^\infty(Y)$, where we view $L^\infty(Y) \subset L^\infty(Z)$ as functions that only depend on the second variable. To deduce that $\cR$ is amenable, choose a Borel right inverse $x \mapsto (\vphi(x),\pi(x))$ for the countable-to-one and surjective Borel map $G \times Y \recht X : (g,y) \mapsto g \cdot y$. Make this choice such that $\vphi(g \cdot y) = g$ and $\pi(g \cdot y) = y$ for all $g \in \cU$ and $y \in Y$. Note that $\pi \colon X \recht Y$ is a factor map, so that we can define the factor map $Z \recht \cR : (x,y) \mapsto (\pi(x),y)$. This factor map induces an inclusion $L^\infty(\cR) \recht L^\infty(Z)$. The composition with $Q$ yields a right invariant mean on $\cR$, i.e.\ a conditional expectation $P \colon L^\infty(\cR) \recht L^\infty(Y)$ satisfying $P(a \cdot \psi) = P(a) \cdot \psi$ for all $a \in L^\infty(\cR)$ and $\psi \in [[\cR]]$. So $\cR$ is amenable.

Conversely assume that $\cR$ is amenable. From Lemma \ref{lem.project-fixed-points}, we get a conditional expectation $Q \colon L^\infty(Z) \recht L^\infty(Z)^\cR$ satisfying $Q(g \cdot a) = g \cdot Q(a)$ for all $a \in L^\infty(Z)$ and $g \in G$. Note that $L^\infty(Z)^\cR = L^\infty(X)$, where we view $L^\infty(X) \subset L^\infty(Z)$ as functions that only depend on the first variable. Composing $Q$ with integration w.r.t.\ $\mu$, we find a $G$-invariant mean on $L^\infty(Z)$. Using the isomorphism $\Psi$ given by \eqref{eq.first-push}, we find a $G$-invariant mean on $L^\infty(G \times Y)$. The restriction to $L^\infty(G) \ot 1$ yields a left invariant mean on $L^\infty(G)$ so that $G$ is amenable.

This concludes the proof of Proposition \ref{prop.properties-cross-section-eq-rel}.

\end{document}